\documentclass{article}
\usepackage[utf8]{inputenc}
\usepackage{enumerate}
\usepackage{amssymb, bm}
\usepackage{amsthm}
\usepackage{amsmath}
\usepackage{tikz-cd} 
\usepackage{bbold}
 \usepackage[all]{xy}
\title{Intersection theory and Chern classes in Bott-Chern cohomology}
\author{Xiaojun WU}
\date{\today}
\newtheorem{mythm}{Theorem}
\newtheorem{mylem}{Lemma}
\newtheorem{myprop}{Proposition}
\newtheorem{mycor}{Corollary}
\newtheorem{mydef}{Definition}
\newtheorem{myrem}{Remark}
\begin{document}
\def\cI{\mathcal{I}}
\def\Z{\mathbb{Z}}
\def\Q{\mathbb{Q}}  \def\C{\mathbb{C}}
 \def\R{\mathbb{R}}
 \def\N{\mathbb{N}}
 \def\H{\mathbb{H}}
  \def\P{\mathbb{P}}
 \def\rC{\mathrm{C}}
  \def\d{\partial}
 \def\dbar{{\overline{\partial}}}
\def\dzbar{{\overline{dz}}}
\def \ddbar {\partial \overline{\partial}}
\def\cB{\mathcal{B}}
\def\cD{\mathcal{D}}  \def\cO{\mathcal{O}}
\def\cbarO{\overline{\mathcal{O}}}
\def\D{\mathcal{D}}
\def\cC{\mathcal{C}}
\def\cF{\mathcal{F}}
\def \rank{\mathrm{rank}}
\def \deg{\mathrm{deg}}
\def \tot{\mathrm{Tot}}
\def \Im{\mathrm{Im}}
\def \id{\mathrm{id}}
\bibliographystyle{plain}
\def \End{\mathrm{End}}
\def \dim{\mathrm{dim}}
\def \div{\mathrm{div}}
\def \ker{\mathrm{Ker}}
\def \rC{\mathrm{Cone}}
\newcommand{\Ub}{\mathcal{U}}
\newcommand{\dcech}{\check{\delta}}
\newcommand{\dcechg}{\delta\!\!\!\check{\delta}}
\newcommand{\lc}{\mathcal{L}}
\newcommand{\ec}{\mathcal{E}}
\def\cB{\mathcal{B}}
\def\cbarO{\overline{\mathcal{O}}}
\def\cD{\mathcal{D}}
\def\D{\mathcal{D}}
\def \rC{\mathrm{Cone}}
\def \id{\mathrm{id}}
\def \ch{\mathrm{ch}}
\def \Td{\mathrm{Td}}
\def \Pic{\mathrm{Pic}}
\def \loc{\mathrm{loc}}
\def \pr{\mathrm{pr}}
\def \Cone{\mathrm{Cone}}
\def \Sh{\mathrm{Sh}}
\def \Ann{\mathrm{Ann}}
\def \Hom{\mathrm{Hom}}
\def \Mod{\mathrm{Mod}}
\def \Tot{\mathrm{Tot}}
\maketitle
\begin{abstract}
In this article, we study the axiomatic approach of Grivaux in \cite{Gri} for rational Bott-Chern cohomology, and use it in particular to define Chern classes of coherent sheaves in rational Bott-Chern cohomology.
This method also allows us to derive a Riemann-Roch-Grothendieck formula for a projective morphism between smooth complex compact manifolds.
\end{abstract}
In the general case of complex spaces, the Poincar\'e and Dolbeault-Grothendieck lemmas are not valid in general. For this reason, and to simplify the exposition, we only consider non-singular complex spaces in the sequel and let $X$ denote throughout a complex manifold. 
\section{Introduction}
Chern classes and Chern characteristic classes are very important topological invariants of complex vector bundles. In order to better reflect the complex structure of manifolds, we refine Chern classes and Chern characteristic classes, and define them in rational Bott-Chern cohomology. This is done by introducing suitable complexes of sheaves of holomorphic and anti-holomorphic forms. There exists a canonical morphism from the complex of rational Bott-Chern cohomology into the locally constant sheaf $\Q$, seen as a complex with a single term located in degree $0$. Under this morphism, the image of Chern classes and Chern characteristic classes in rational Bott-Chern cohomology are the usual ones defined in singular cohomology.

In the fundamental article \cite{Gri}, Grivaux showed that for suitable rational cohomology theories of compact complex manifolds, one could construct Chern characteristic classes of arbitrary coherent sheaves, and in particular of torsion sheaves, by induction on the dimension. This can be done provided one has a reasonable intersection theory, and provided Chern classes can be defined for vector bundles. One important argument consists of ensuring the validity of the Riemann-Roch-Grothendieck formula for closed immersions of smooth hypersurfaces.
In \cite{Gri}, he proves that his axioms hold for rational Deligne cohomology and hence constructs Chern characteristic classes in rational Deligne cohomology.

We begin by recalling some background for this type of problem.
For any complex manifold $X$, we denote by $K^0X$ the Grothendieck group of vector bundles on $X$. For a vector bundle $E$, we denote by $[E]$ the class represented by $E$. By definition, $K^0X$ is the quotient of the free abelian group on the set of isomorphism classes of vector bundles, modulo the relations
$$[E]=[E']+[E'']$$
for all exact sequences $0 \to E' \to E \to E'' \to 0$. It can be endowed with a ring structure by taking tensor products of vector bundles.

In a similar way, we denote by $K_0X$ the Grothendieck group of coherent sheaves on $X$, simply by replacing vector bundles in the definition of $K^0X$ by coherent sheaves, and one has a natural morphism $K^0X\to K_0X$ by viewing vector bundles as coherent sheaves. This morphism is an isomorphism in the projective case. However, by the fundamental work of Voisin \cite{Voi}, $K^0X$ can be strictly smaller than $K_0X$ when $X$ is a compact K\"ahler manifold. This phenomenon is caused by the lack of global resolutions of coherent sheaves by locally free sheaves.

Over $\Q$, Chern characteristic class can be seen through the $\Q$-linear morphism
$$\ch: K^0(X)\otimes_{\Z} \Q \to A(X), $$
where $A(X)$ means the cohomology ring in the cohomology theory under consideration. A priori,  on arbitrary compact complex manifolds, it is not trivial that this morphism can extended into a morphism from $K_0(X)\otimes_{\Z} \Q$. Grivaux showed that this is possible once the cohomology theory satisfies suitable axioms of intersection theory. The aim of this note is to develop a similar intersection theory for integral (or rational) Bott-Chern cohomology.

Such theories have been considered in the work \cite{Sch} of M.~Schweitzer, and have also been developed in a more recent unpublished work of Junyan Cao.
They are more precise than Deligne cohomology or complex Bott-Chern cohomology, in the sense that there always exist natural morphisms from the integral (or rational) Bott-Chern cohomology into the other ones. We use here Grivaux's axiomatic approach to construct Chern classes in rational Bott-Chern cohomology, for coherent sheaves on arbitrary compact complex manifolds. 

In fact, it would be interesting to give a construction of Chern classes of coherent sheaves in the integral Bott-Chern cohomology rather than the rational one, but substantial difficulties remain.
Let $\cF$ be a coherent sheaf on a smooth hypersurface $D$ of $X$. We denote by $i: D \to X$ the inclusion. One of the main difficulties is to express the total Chern class $c(i_* \cF)$ in function of  $i_* c_{\bullet}(\cF)$ and $i_* c_{\bullet}(N_{D/X})$, where $N_{D/X}$ is the normal bundle of $D$ in $X$. 
There exists a formulation of the Riemann-Roch-Grothendieck formula that does not involve denominators, but it does not seem to be easily applicable since Chern classes of coherent sheaves, unlike in the vector bundle case, may involve data in higher degrees than the generic rank.

For the convenience of the reader, we summarise here the axioms needed in the axiomatic cohomology theory developed in \cite{Gri}. We assume that for any compact complex manifold $X$ we can associate to $X$ a graded commutative cohomology ring $A(X)$ which is also a $\Q (\subset A^0(X))$-algebra.
In the following, we will focus on the rational (or integral) Bott-Chern cohomologies.
\paragraph{}
\textbf{Axiom A} (Chern classes for vector bundles)
\\
(1) For each holomorphic map $f:X \to Y$, there exists a functorial pull-back morphism $f^*: A(Y) \to A(X)$ which is
compatible with the products and the gradings (by construction, cf. Section 2 and Proposition 5).
\\
(2) One has a group morphism $c_1: \Pic(X) \to A^1(X)$ which is compatible with pull-backs (by construction, cf. Section 4).
\\
(3) (Splitting principle) If $E$ is a holomorphic vector bundle of rank $r$ on $X$, then $A(\P(E))$ is a free graded module over $A(X)$ with basis $1,c_1(\cO_E(1)),\cdots,(c_1(\cO_E(1))^{r-1}$ (cf. Proposition 7).
\\
(4) (homotopy principle) For every $t$ in $\P^1$,
let $i_t$ be the inclusion $X \times \{t\} \hookrightarrow X \times \P^1$. 
Then the induced pull-back morphism $i_t^*: A(X \times \P^1) \to A(X)\cong A(X \times \{t\})$ is independent of $t$ (cf. Lemma 6).
\\
(5) (Whitney formula) Let $0 \to E \to F \to G \to 0$ be an exact sequence of vector bundles, then one has $c(F)=c(E) \cdot c(F)$ and $\ch(F)=\ch(E)+\ch(G)$ where $c(E)$ means the total Chern class of $E$ and $\ch(E)$ means the Chern characteristic class of $E$ (cf. Proposition 8).

The construction of the pull-back will be given in the second section, and the other parts are important results of Junyan Cao which will be given in the fourth section.

\textbf{Axiom B} (Intersection theory) 

If $f:X \to Y$ is a proper holomorphic map of relative dimension $d$,
there is a functorial Gysin morphism $f_*: A^{\bullet}(X) \to A^{\bullet -d}(Y)$
satisfying the following properties:
\\
(1) (Projection formula) For any $x \in A(X)$ and any $y \in A(Y)$ one has $f_*(x \cdot f^* y)=f_*(x) \cdot y$ (cf.  Proposition 6).
\\
(2) Consider the following commutative diagram with $p,q$ the projections on the first factors
$$  \xymatrix{
    Y \times Z \ar@{^{(}->}[r]^{i_{Y \times Z}} \ar[d]_p & X \times Z \ar[d]^q \\
    Y \ar@{^{(}->}[r]^{i_Y} & X
  }$$
Assume $Z$ to be compact and $i_Y$ proper. Then one has
$ i_{Y}^*q_*=p_* i_{Y \times Z}^*$ (cf. Proposition 1) .
\\
(3) Let $f: {X} \to {Y}$ be a surjective proper map
between compact manifolds, and let $D$ be a
smooth divisor of $Y$. 
We denote $f^* D=m_{1}\tilde{D}_1+\cdots +m_{N}\tilde{D}_{N}$ with $\tilde{D}_i$ simply normal crossing. Let
$\tilde{f}_i: \tilde{D}_{i} \to {D}$ ($1 \leq i \leq N $) be
the restriction of $f$ to
$\tilde{D}_{i}$. Then one has
$$
f^* i_{D*}=\sum _{i=1}^{N} m_{i}\,
i_{\tilde{D}_{i}*} \tilde{f}_{i}^{*}
$$
(cf. Proposition 9).
\\
(4) Consider the commutative diagram, where
  $Y$ and $Z$ are compact and intersect transversally with $W=Y \cap Z$:
$$
\xymatrix{
W  \ar@{^{(}->}[r]^{i_{W/Y}}
\ar@{_{(}->}[d]_{i_{W/Z}}
&Y\ar@{_{(}->}[d]^{i_{Y}}
\\
Z \ar@{^{(}->}[r]_{i_{Z}}&X.
}
$$
Then one has
$i_{Y}^* i_{Z*}=i_{W/Y*}i^*_{W/Z}$ (cf. Proposition 10).
\\
(5)(Excess formula) If $Y$ is a smooth hypersurface of a compact complex manifold $X$, then for any cohomological class $\alpha$ we have
$$i_Y^* i_{Y*} \alpha=\alpha \cdot c_1(N_{Y/X})$$
(cf. Proposition 11).
\\
(6) The Hirzebruch–Riemann–Roch theorem holds for $(\P^n,\cO(i))$ ($\forall i$) (cf. Proposition 14).
\\
(7) Let $X$ be a compact complex manifold with $\mathrm{dim}_{\C} X = n$ and $Y \subset X$ be a closed complex submanifold of complex codimension $r \geq 2$. Suppose that $p : \tilde{X}\to X$ is the
blow-up of $X$ along $Y$. We denote by $E$ the exception divisor and $i: Y \to X$, $j:E \to \tilde{X}$ the inclusions, and $q: E \to Y$ the restriction of $p$ on $E$. Then $p^*$ is injective (cf. Lemma 3) and 
there is an isomorphism induced by $j^*$
$$j^*:A^{\bullet}(\tilde{X})/p^*A^{\bullet}(X) \cong A^{\bullet}(E)/q^*A^{\bullet}(Y).$$
In other words, a class $\alpha \in A^{\bullet}(\tilde{X})$ is in the image of $p^*$
 if and only if the class $j^* \alpha$ is in the image of $q^*$ (cf. Proposition 12).
If $F$ is the excess conormal bundle on $E$ defined by the exact sequence
$$0 \to F \to q^*N^*_{Y/X} \to N^*_{E/\tilde{X}}\to 0,$$
one has the following excess formula for any cohomology class $\alpha$ on $Y\;:$
$$p^* i_* \alpha=j_* (q^* \alpha \cdot c_{d-1}(F^*))$$
(cf. Proposition 13).
\paragraph{}
In the parentheses, for the convenience of the readers, we state the corresponding verification for the integral Bott-Chern cohomology.
Note that in the paper of Grivaux, he also lists Axiom C for a complete characterization of a theory of Chern classes of coherent sheaves.
However, Axiom A and B imply Axiom C.
For this reason, we omit the list of Axiom C.

Note also that in the published version \cite{Gri}, Axiom A (4) is deleted, but it is needed to prove Whitney formula, so we still state it explicitly.
Also Axiom A (5) is different from Axiom A (4) in \cite{Gri}.
It is easy to see that Axiom A (4) in \cite{Gri} implies Whitney formula in the case that the exact sequence splits.
However, Whitney formula is what is really needed to define uniquely the Chern classes of vector bundles in Grothendieck's axiomatic approach and also for the rest of construction \cite{Gri} instead of Axiom A (4) in \cite{Gri}.
It does not seem apparent to the author that it is trivial to deduce Whitney formula in rational Bott-Chern cohomology for the non-splitting case from the splitting case.
To keep in form of the usual axiomatic definition of Chern classes of vector bundles, we state it in form of Whitney formula instead of Axiom A (4) in \cite{Gri}.

The verification of axiom B will constitute the main substance of the fifth and sixth sections. In principle,  pull-backs can be induced by taking the pull-back of smooth forms, and push-forwards can be induced by taking the push-forward of currents under proper morphisms. The proof of the first two axioms is then reduced to considering the natural pairing between smooth forms and currents. The third and fourth axioms are more complicated, since they demand taking pull-backs of currents. As in the case of Deligne cohomology, we first reduce the situation to the case of cycle classes. Then we reduce the verification of properties of cycle classes in integral Bott-Chern cohomology to the corresponding properties of Deligne cycle classes.
Checking the remaining axioms is more standard. This will be done in the sixth section.

One difference between the above axiom compared to that of \cite{Gri} should be pointed out.
It appears that axiom B (2) is incorrectly
formulated in the paper of Grivaux for the construction of Chern classes, and that it should actually be formulated as in the present
article.
Moreover, to prove axiom B (3), one needs a generalised version of axiom B (2)
where the inclusion $i : Y \to X$ should be replaced by a
proper map $f : Y \to X$.
Since this is the only point that we need this generalised version, we still form our axiom for closed immersions, but this should be clarified.

In conclusion, we have the following result.
\begin{mythm}
For a compact complex manifold $X$,
the cohomology ring $\oplus_k H^{k,k}_{BC}(X, \Q)$ satisfies axiom A, B. In fact, the cohomology ring $\oplus_k H^{k,k}_{BC}(X, \Z)$ satisfies axiom A, B except for the sixth one of list B, which demands rational coefficients to define Chern characteristic classes and the Todd class.
\end{mythm}
As a consequence, by the work of \cite{Gri}, for the rational Bott-Chern cohomology, we get the following result.
\begin{mycor}
If  $X$  is compact and $K_0X$
is the Grothendieck ring of coherent sheaves on $X$ , one can define a Chern character morphism $\ch: K_0X \to \oplus_k H^{k,k}_{BC}(X, \Q)$
such that
\\
(1)  the Chern character morphism is functorial by pull-backs of holomorphic maps.
\\
(2)  the Chern character morphism is an extension of the usual Chern character morphism for locally free sheaves given in axiom A.
\\
(3) The  Riemann–Roch-Grothendieck theorem holds  for  projective  morphisms between smooth complex compact manifolds.
In other words, let $p: X \to S$ a projective morphism of compact complex manifolds and $\cF$ be a coherent sheaf over $X$. Then we have the Riemann-Roch-Grothendieck formula in the rational and complex Bott-Chern cohomology.
\end{mycor}
The rational case is a direct consequence of the work of \cite{Gri}, which uses classical arguments of Serre to reduce the proof to the fact that the Riemann-Roch-Grothendieck formula holds for a closed immersion. It is proven by construction of Chern characteristic classes (or equivalently of Chern classes in the rational coefficient case), using the prescribed axioms of intersection theory. 

Let us discuss the case of complex Bott-Chern cohomologies.
Using the methods developed in this note combined with the work of \cite{Gri}, we give as an application a more algebraic proof of the Riemann–Roch-Grothendieck theorem of Bismut \cite{Bis1}, \cite{Bis2} under the additional assumption that the morphism is projective. However, we do not need the condition that the sheaf and all of its direct images are locally free, nor the condition that the morphism is a submersion (cf. also \cite{BSW21}). 
The complex case can be derived from the rational case by the natural morphism from the rational Bott-Chern cohomology to the complex Bott-Chern cohomology.
In the rest of the paper, we will focus on the rational (or integral) case.

The organisation of the paper is the following. Section two recalls basic definitions and introduces pull back and push forward morphisms. Section three introduces a ring structure on the integral Bott-Chern cohomology, in such a way that it is compatible with the ring structure of the complex Bott-Chern cohomology via the canonical map.
Section four gives the construction of Chern classes associated with a vector bundle and verifies the list of axioms A. 
Section five introduces cycle classes in integral Bott-Chern cohomology and verifies the intersection theory part of axioms B.
Section six studies the transformation of Chern classes under blow ups. This completes the verification of axioms B.

\textbf{Acknowledgement.}  
I thank Jean-Pierre Demailly, my PhD supervisor, for his guidance,patience and generosity.
 I am indebted to St\'ephane Guillermou,  Julien Grivaux, Honghao Gao and Bingyu Zhang for very helpful suggestions, in particular, St\'ephane Guillermou for the help on the general theory of sheaves. I would like to thank Junyan Cao for authorising using one of his unpublished work. I would also like to express my gratitude to colleagues of Institut Fourier for all the interesting discussions we had. During the course of this research, my work has been supported by a Doctoral Fellowship AMX attributed by \'Ecole Polytechnique and Ministère de l’Enseignement Supérieur et de la Recherche et de l’Innovation de France, and I have also benefited from the support of the European Research Consortium grant ALKAGE Nr. 670846 managed by J.-P. Demailly.
 We thank the anonymous reviewer for a very careful reading of this paper, and for insightful comments and suggestions.
\section{Definition of integral Bott-Chern cohomology classes}
In this section, we recall the basic definitions associated with integral Bott-Chern cohomology. A reference for this part is \cite{Sch}. Notice that changing $\Z(p)$ by $\C$ in the integral Bott-Chern complex gives a complex which defines the complex Bott-Chern cohomology. Hence one gets a canonical map from the integral Bott-Chern cohomology to the complex Bott-Chern cohomology.
Next, we define pull backs and push forwards in integral Bott-Chern cohomology. 
We verify the axioms without involving the ring structure of the integral Bott-Chern cohomology (namely Axiom B (2), part of (7)).
\begin{mydef}
The integral Bott-Chern cohomology group is defined as the hypercohomology group 
$$H^{p,q}_{BC}(X,\Z)=\H^{p+q}(X, \cB^*_{p,q, \Z})$$
of the integral Bott-Chern complex
\begin{equation}
\cB^{\bullet}_{p,q,\Z}: \Z(p) \xrightarrow{\Delta} \cO \oplus \cbarO  \to \Omega^1 \oplus \overline{\Omega^1} \to \cdots \to \Omega^{p-1} \oplus \overline{\Omega^{p-1}} \to \overline{\Omega^{p}} \to \cdots \to \overline{\Omega^{q-1}} \to 0
\end{equation}
where $\Z(p)=(2 \pi i)^p\Z$ at 0 degree and $\Delta$ is multiplication by 1 for the first component and multiplication by -1 for the second component. We call rational (or complex) Bott-Chern cohomology the hypercohomology of the complex obtained by changing $\Z(p)$ respectively into $\Q, \C$.
\end{mydef}
Notice that the choice of the sign in $\Delta$ is to ensure that the natural map from the integral Bott-Chern cohomology to the complex Bott-Chern cohomology is a ring morphism.
This will be discussed in Section~3.
The choice of $\Z(p)$ instead of $\Z(q)$ is more or less artificial, but since the Chern class always lies in $H^{p,p}_{BC}(X, \Z)$ for some $p$, this choice poses no problem.

We begin by the definition of pull-backs of cohomology classes.
Let $f: X \to Y$ be a holomorphic map, it induces a natural morphism of complexes of abelian group on any open set $U$ of $Y$, $\cB^{\bullet}_{p,q,\Z,Y}(U) \xrightarrow{f^*}\cB^{\bullet}_{p,q,\Z,X}(f^{-1}(U))$ which induces the cohomological class morphism $H^{p,q}_{BC}(Y,\Z) \xrightarrow{f^*} H^{p,q}_{BC}(X,\Z)$. More precisely, the pull-back of forms induces a morphism of complexes $f^* \cB^{\bullet}_{p,q,\Z,Y} \xrightarrow{f^*}\cB^{\bullet}_{p,q,\Z,X}$ on $X$ which induces a cohomological morphism $\H^{\bullet}(X,f^* \cB^{\bullet}_{p,q,\Z,Y}) \xrightarrow{}\H^{\bullet}(X,\cB^{\bullet}_{p,q,\Z,X})$. 
On the other hand, there exists a natural morphism $\H^{\bullet}(Y,\cB^{\bullet}_{p,q,\Z,Y}) \xrightarrow{}\H^{\bullet}(X,f^{*}\cB^{\bullet}_{p,q,\Z,Y})$ since the pre-image of any open covering of $Y$ gives an open covering of $X$. 
The composition of two morphisms gives the pull back morphism $H^{p,q}_{BC}(Y,\Z) \xrightarrow{f^*} H^{p,q}_{BC}(X,\Z)$.
The second morphism can be interpreted more formally as follows. There exists a natural morphism $\cB^{\bullet}_{p,q,\Z,Y} \to Rf_* f^* \cB^{\bullet}_{p,q,\Z,Y}$. Taking $R\Gamma(Y,-)$ on both sides gives $\H^{\bullet}(Y,\cB^{\bullet}_{p,q,\Z,Y}) \xrightarrow{}\H^{\bullet}(X,f^{*}\cB^{\bullet}_{p,q,\Z,Y})$.
\paragraph{}
For a proper holomorphic map $f: X \to Y$ of relative dimension $d$, we next construct a functorial Gysin morphism $f_*: H^{p,q}_{BC}(X,\Z) \to H^{p-d,q-d}_{BC}(Y,\Z)$. The construction is a modification of the similar construction for Deligne cohomological class given in \cite{ZZ}. The condition of properness is necessary even if we just consider cycle classes, since the image of an analytic set is not necessarily an analytic set when the properness condition is omitted.

Let $K^{\bullet}$ be a complex of sheaves on the space $X$. One denotes by $\{F^pK^{\bullet}\}$ the stupid filtration which does not preserve the cohomology at degree $p$ i.e. if $q \geq p$, $F^pK^{q}=K^q$, otherwise $F^pK^{q}=0$.
For the corresponding quotient complex, we denote it as $\sigma_pK^{\bullet}=K^{\bullet} /F^pK^{\bullet}$.
We denote by $\Omega^{\bullet}$ the complex of sheaves of holomorphic forms on $X$. Let $i: \Z(p) \to \sigma_p\Omega^{\bullet} \oplus \sigma_{q}\overline{\Omega^{\bullet}}$ be the complex map defined by the diagonal map sending $\Z(p)$ into $\cO_X \oplus \overline{\cO_X}$ at degree 0 with a sign $-1$ at the second component and viewing $\Z(p)$ as a complex centred at degree 0. 
With the above notations, the integral Bott-Chern complex is the mapping cone of $i$ which we denote as $\rC^{\bullet}(i)[-1]$. The idea to define the push-forward of the cohomology class is to choose compatible resolutions of the complexes $\Z(p), \sigma_p\Omega^{\bullet} \oplus \sigma_{q}\overline{\Omega^{\bullet}}$ such that the both complexes are formed by some kind of currents for which the push-forward is well-defined. 
\paragraph{}
For the convenience of the readers, we recall here some basic definitions and properties concerning currents and geometric measure theory.
We will use them to define a resolution of $\Z(p)$.
For more details and proofs, we refer to the article of \cite{Ki}.
We need locally integral currents to construct a resolution of locally constant sheaf.
\begin{mydef}
The space of locally integral currents is defined by
$$\cI_r^{\mathrm{loc}}(N):=\{T \in R_r^{\mathrm{loc}}(N)| dT\in R_r^{\mathrm{loc}}(N) \}$$
where $R_r^{\mathrm{loc}}$ is the sheaves of locally rectifiable currents.
\end{mydef}
We donot give the precise definition of locally rectifiable currents (generalized singular
chains with integer coefficients) here since these
definitions by themselves play no role in the article, just certain properties
of locally integral currents.
We refer to the book of \cite{Fe} for more information.

\begin{mylem}
The complex of locally integral currents gives a soft resolution of the locally constant sheaf $\Z$ over a manifold.
\end{mylem}
\begin{proof}
It is enough to observe the fact that 
for $T \in \cI_m^{\mathrm{\loc}}(\R^n)$ such that $dT=0$ there exists a $S \in \cI_{m+1}^{\mathrm{loc}}(\R^n)$ such that $dS=T$
(cf. \cite{Fe} 4.2.10 as a consequence of the deformation theorem) and the following proposition in \cite{Ki} proposition 2.1.9 for the case of top degree.

The fact that the sheaves of locally integral currents are soft can be found in the discussions on Page 57 \cite{HK74} before Theorem 2.2 (as a consequence of Federer theory of slicing). 
\end{proof}
\begin{mythm}
Let $M$ be a Riemannian manifold of dimension n. If $T \in \cD^{'0}(M)$ such that $dT=0$ then $T$ is the current defined by locally constant functions.
If $T \in  \cI_n^{\mathrm{loc}}(M)$ then this function is integral valued. 
\end{mythm}
We now return to the construction of the push forward for hypercohomology.
We denote by ${\cD_X'}^{p,q}$ the sheaf of currents of type $(p,q)$ on $X$. For each $p$, $({\cD_X'}^{p,\bullet}, \dbar)$ is a fine resolution of $\Omega^p_X$. By taking the conjugation, $({\cD_X'}^{\bullet,q}, \d)$ is a fine resolution of $\overline{\Omega^q_X}$. The conjugate of differential forms induces the conjugate of currents. In particular, $\sigma_{p,\bullet} {\cD_X'}^{\bullet,\bullet}$ (resp. $\sigma_{\bullet,q} {\cD_X'}^{\bullet,\bullet}$) is a Cartan-Eilenberg resolution of $\sigma_p\Omega^{\bullet}_X$ (resp. $\sigma_q \overline{\Omega^{\bullet}_X}$). Taking the total complex of the double complex, we deduce that
$\sigma_p {\cD_X'}^{\bullet}$ is a resolution of $\sigma_p\Omega^{\bullet}_X$. Here, we use an abuse of notation, and actually mean that we take direct sums of spaces of currents of bidegree $(k,l)$ with $k \leq p$. Similarly, $\sigma_q {\cD_X'}^{\bullet}$ is a resolution of $\sigma_q \overline{\Omega^{\bullet}_X}$. 
By taking complex coefficients, locally integral currents extend into a complex of $\C$-vector spaces of currents instead of $\Z$-modules.

Let $\cI^i_X$ be sheaf of $\Z(p)$ times locally integral currents of real codimension $i$ on $X$, as defined above.
In the following, to simplify the notations, we make an abuse of notation to ignore the factor $\Z(p)$.
The complex $\cI^{\bullet}_X$ is a soft resolution of $\Z$ by Lemma 1. The integral Bott-Chern complex is quasi-isomorphic to the following complex obtained by composing the natural inclusion of forms into currents:
$$\Z(p)  \xrightarrow{\Delta} \sigma_p {\cD_X'}^{\bullet} \oplus \sigma_q {\cD_X'}^{\bullet}. $$
This morphism of complexes factorises into
\begin{equation}
 \Z(p) \to \cI_X^{\bullet}  \xrightarrow{\Delta} \sigma_p {\cD_X'}^{\bullet} \oplus \sigma_q {\cD_X'}^{\bullet}.
\end{equation}
In the following, we denote $\tilde{\cB}_{X}^{\bullet}$ the mapping cone of $\Delta$ in (2).
We will call it the integral Bott-Chern complex involving integral currents and currents.

The morphism of complexes $\Delta$ factorises itself into the composition of two maps : the first is the diagonal map with positive sign on the first component and negative sign on the second component with image in $ {\cD_X'}^{\bullet} \oplus   {\cD_X'}^{\bullet}$; 
the second map is the decomposition of locally integral currents into their components of adequate bidegrees.

Since the first inclusion is a quasi-isomorphism in the derived category in $D(\Sh(X))$, the integral Bott-Chern complex is quasi-isomorphic to $\rC^{\bullet}(\Delta)[-1]: \cI_X^{\bullet}  \xrightarrow{\Delta} \sigma_p {\cD_X'}^{\bullet} \oplus \sigma_q {\cD_X'}^{\bullet}. $ Note that the push-forward of currents and of the locally integral currents are both well-defined for a proper morphism. We also remark that the rule $df_*=f_*d$ holds for currents. Hence there exists a natural morphism of complexes on $Y$ 
$$f_*\cI_X^{\bullet} \to \cI_Y^{\bullet-d}, f_*(\sigma_p {\cD_X'}^{\bullet} \oplus \sigma_q {\cD_X'}^{\bullet}) \to \sigma_{p-d} {\cD_Y'}^{\bullet} \oplus \sigma_{q-d} {\cD_Y'}^{\bullet}$$
which, as will be explained below, induces a cohomological group morphism 
$$f_*:H^{p,q}_{BC}(X,\Z) \to H^{p-d,q-d}_{BC}(Y,\Z).$$
Here, to define the push-forward for cohomology classes, it is enough to define it for global section representatives; in fact, the complex $\cI_X^{\bullet}$ is soft, which means any section over any closed subset can be extended to a global section; a soft sheaf is in particular acyclic, thus the complex $\sigma_p {\cD_X'}^{\bullet} \oplus \sigma_q {\cD_X'}^{\bullet}$ is acyclic.
The hypercohomology of the integral Bott-Chern complex is just the cohomology of the global sections of the mapping cone $\Delta$. 
Now we define the push-forward of a cohomology class as the push-forward of any of the global currents representing the cohomology class.
By construction, the pull-back and push-forward both satisfy the functoriality property.

Notice that the use of a resolution of the locally constant sheaf $\Z(p)$ seems to be necessary since a priori we have only natural morphism in inverse direction $\H^{\bullet}(Y,f_* \cB^{\bullet}_{p,q,\Z,X}) \xrightarrow{}\H^{\bullet}(X,\cB^{\bullet}_{p,q,\Z,X})$. The trace morphism $\mathrm{tr}: f_*\Z_X \to \Z_Y$ and the push forward of currents induces a morphism $\H^{\bullet}(Y,f_* \cB^{\bullet}_{p,q,\Z,X}) \xrightarrow{}\H^{\bullet}(Y,\cB^{\bullet}_{p,q,\Z,Y})$ if $X,Y$ have the same dimension. It seems to be not easy to induces from these two morphisms a morphism
$\H^{\bullet}(X, \cB^{\bullet}_{p,q,\Z,X}) \xrightarrow{}\H^{\bullet}(Y,\cB^{\bullet}_{p,q,\Z,Y})$. 
If we take the quasi-isomorphic acyclic resolution involving the locally integral currents as in (2), the hypercohomology of $\H^{\bullet}(X,B_{p,q,\Z,X}^{\bullet})$ is represented by global sections. Then the restriction of the global section on the open sets induces a morphism $\H^{\bullet}(X, \cB^{\bullet}_{p,q,\Z,X}) \xrightarrow{}\H^{\bullet}(Y,f_* \cB^{\bullet}_{p,q,\Z,X})$ in the desired direction. 
In this case, we have the following factorisation
$$\begin{tikzcd}
\H^{\bullet}(X, \cB^{\bullet}_{p,q,\Z,X})
\arrow[rd,"f_*"]
\arrow[r] & \H^{\bullet}(Y,f_* \cB^{\bullet}_{p,q,\Z,X})\arrow[d]\\
& \H^{\bullet}(Y, \cB^{\bullet}_{p,q,\Z,Y}[-2d])
\end{tikzcd}
$$
where $d$ is the relative complex dimension.
The vertical arrow is the morphism induced by pushing forward currents, under the assumption that $f$ is proper.

Commutativity can be checked directly. Let $T$ be the global section representing a cohomology class in $\H^{\bullet}(X, \cB^{\bullet}_{p,q,\Z,X})$. Let $(V_i)_i$ be an open Stein covering of $Y$ such that the hypercohomology class on $Y$ can be calculated by the hypercohomology associated with the open cover. 
We denote by $\{T_i\}$ the image of $T$ in $\H^{\bullet}(Y,f_* \cB^{\bullet}_{p,q,\Z,X})$ by restriction on $V_i$. More precisely, $T_i$ is the restriction of $T$ on $f^{-1}(V_i)$. Its image in $\H^{\bullet}(Y, \cB^{\bullet}_{p,q,\Z,Y}[-2d])$ is $\{f_* T_i\}$, and those sections glue into a global section $f_*T$.

The definition of the push-forward of cohomology classes can also be interpreted more formally as follows. In order to distinguish the different morphisms of complexes, we denote by $\Delta_X$ the map on $X$ involving $\Z(p)$ and $\tilde{\Delta}_X$ the map on $X$ involving locally integral currents. 
The complex $\rC(\tilde{\Delta}_X)$ involving locally integral currents is a soft complex. Since $f$ is proper, $f_* \rC(\tilde{\Delta}_X)$ is a soft complex which means $f_* \rC(\tilde{\Delta}_X)=Rf_* \rC(\Delta_X)$ in $D(\Sh(Y))$. We denote by $a_X$ (resp. $a_Y$) the morphism from $X$ (resp. $Y$) to a point.
The push forward of currents induces a morphism of complexes in $C(\Sh(Y))$: $f_* \rC(\tilde{\Delta}_X) \to \rC(\tilde{\Delta}_Y)[-2d]$. In other words, we have by composition a morphism in the derived category 
$$Rf_*(\rC(\Delta_X)) \to \rC(\Delta_Y)[-2d].$$
Taking $R\Gamma(Y,-)=Ra_{Y*}$ on both sides, and using the fact that $R(a_Y \circ f)_*=Ra_{X*}=Ra_{Y*} \circ Rf_{*}$ (since $f_*$ transforms soft complexes into soft complexes), we get $f_*:H^{p,q}_{BC}(X,\Z) \to H^{p-d,q-d}_{BC}(Y,\Z)$
after taking cohomology.

In the following, once we want to view the push forward of the cohomology groups as a morphism in the cohomology level induced by a morphism of complexes, we use the above interpretation (for example, in the proof of the projection formula). 
\paragraph{}
In the case where $f$ is analytic fibration, in the sense that $f$ is a proper surjective morphism and all fibres are connected, we can additionally define a morphism from the push forward of the locally constant sheaf $\Z_X$ to the locally constant sheaf $\Z_Y$, e.g.\ a morphism $f_* \Z_X \to \Z_Y$. Any modification $f$ such as a composition of blows-up with smooth centers is an example of an analytic fibration in the above setting. We now use this morphism to prove that any modification $p$ yields an injective morphism $p^*$ between the corresponding integral Bott-Chern cohomology groups.

In this case, for any connected open set $V \subset Y$, we have $f_*\Z_X(V)=Z_X(f^{-1}(V))$ where $f^{-1}(V)$ is a connected open set, so it is enough to define the morphism $f_*\Z_X \to \Z_Y$ by asserting that it associates the constant function 1 on $f^{-1}(V)$ to the constant function 1 on $V$. 
In preparation for the next steps, we need the following lemma.
\begin{mylem}
For any analytic fibration $f: X \to Y$, there is a commutative diagram
\[ \begin{tikzcd}
f_*\Z_X \arrow{r}{} \arrow[swap]{d}{} & f_*\cI_X^0 \arrow{d}{} \\%
\Z_Y \arrow{r}{}& \cI_Y^0.
\end{tikzcd}
\]
\end{mylem} 
\begin{proof}
This is directly verified on any connected open set $V \subset Y$. The map $\Z_X(f^{-1}(V)) \to  \cI_X^0(f^{-1}(V))$ is given by associating the constant function 1 to the integral current $[f^{-1}(V)]$ associated with $f^{-1}(V)$. The image of the constant function 1 under $\Z_X(f^{-1}(V)) \to  \Z_Y(V)$ is the constant function 1 on $V$. The image of the constant function 1 under $\Z_Y(V) \to  \cI^0_Y(V)$ is the integral current $[V]$ associated with $V$ which is also the image of $[f^{-1}(V)]$ under $f_*\cI_X^0(V) \to \cI_Y^0(V)$.
\end{proof}
Using an identification of the push forward of currents on $X$ as currents on $Y$, we get the following commutative diagram
\[ \begin{tikzcd}
f_* \rC({\Delta}_X) \arrow{r}{} \arrow[swap]{d}{} & f_*\rC(\tilde{\Delta}_X) \arrow{d}{} \\%
\rC({\Delta}_Y)[-2d] \arrow{r}{}& \rC(\tilde{\Delta}_Y)[-2d]
\end{tikzcd}
\]
with the above notations.
Taking $Ra_{Y*}$ and cohomology to the commutative diagram gives
\[ \begin{tikzcd}
\H^{\bullet}(Y,f_* \cB^{\bullet}_{p,q,\Z,X}) \arrow{r}{} \arrow[swap]{d}{} & \H^{\bullet}(X, \cB^{\bullet}_{p,q,\Z,X}) \arrow{d}{f_*} \\%
\H^{\bullet}(Y, \cB^{\bullet}_{p,q,\Z,Y}[-2d]) \arrow{r}{\id}& \H^{\bullet}(Y, \cB^{\bullet}_{p,q,\Z,Y}[-2d]).
\end{tikzcd}
\]

In the case of a modification, one can prove that $f^*$ is injective. This can be seen via the following
\begin{mylem}For any modification $f:X \to Y$, one has
$$f_*f^*=\id: H_{BC}^{\bullet,\bullet}(Y, \Z) \to H_{BC}^{\bullet,\bullet}(Y, \Z).$$
\end{mylem}
\begin{proof}

Using the above commutative diagram, it is enough to show that for any open set $V \subset Y$ and any sheaf in the integral Bott-Chern complex one has the identity $f_*f^*=\id$, so that the identity will hold for any hypercocycle representing an integral Bott-Chern cohomology class. 

Let $A$ be an analytic set of $X$, $Z$ be an analytic set of $Y$ such that the map $f|_{X \smallsetminus A}: X \smallsetminus A \to Y \smallsetminus Z$ is biholomorphic. 
For any smooth form $\omega$ defined on $V$,
we have $f_*f^* \omega=\omega$. In fact, for any smooth form $\tilde{\omega}$ with compact support in $V$, we can write
$$\langle f_* f^* \omega, \tilde{\omega} \rangle=\langle f^* \omega, f^* \tilde{\omega} \rangle= \int_{f^{-1}V}  f^* \omega \wedge f^* \tilde{\omega}=
\int_{f^{-1}V\smallsetminus A}  f^* \omega \wedge f^* \tilde{\omega}$$
$$=\int_{V\smallsetminus Z}   \omega \wedge  \tilde{\omega}=\int_{V}   \omega \wedge  \tilde{\omega}=\langle \omega, \tilde{\omega} \rangle.$$
Here, the third and fourth equality hold since the integral of a smooth form on an analytic set of lower dimension is 0 (such a set being of Lebesgue measure 0 in the relevant dimension).

For the locally constant sheaf $\Z$, since the analytic fibration has connected fibres, a straightforward argument yields $f_*f^*=\id$.

In conclusion the composition of sheaf morphisms:
$\cB^{\bullet}_{p,q,\Z,Y} \to f_* f^* \cB^{\bullet}_{p,q,\Z,Y}$ (given by the canonical map), $f_* f^* \cB^{\bullet}_{p,q,\Z,Y} \to f_* \cB^{\bullet}_{p,q,\Z,X}$ (induced by pull-back of smooth forms)
and $f_* \cB^{\bullet}_{p,q,\Z,X} \to  \cB^{\bullet}_{p,q,\Z,Y}$ (induced by push-forward of currents)
is the identity map. 
Notice that a priori, the image complex of the last morphism should be the quasi-isomorphic complex involving currents instead of smooth forms. 
However, in the case of a modification, the push forward of a pull-back of a smooth form is still a smooth form.
In particular, the composition of sheaf morphisms $$\cB^{\bullet}_{p,q,\Z,Y} \to f_* f^* \cB^{\bullet}_{p,q,\Z,Y} \to f_* \cB^{\bullet}_{p,q,\Z,X}\to  \cB^{\bullet}_{p,q,\Z,Y}$$
is the identity map. This shows that the canonical map $\cB^{\bullet}_{p,q,\Z,Y} \to f_* f^* \cB^{\bullet}_{p,q,\Z,Y}$
is an isomorphism.

Thus we have the following commutative diagram
$$\begin{tikzcd}
\H^{\bullet}(Y,\cB^{\bullet}_{p,q,\Z,Y})=
\H^{\bullet}(Y,f_*f^*\cB^{\bullet}_{p,q,\Z,Y})
\arrow[d]
\arrow[r] & \H^{\bullet}(Y,f_* \cB^{\bullet}_{p,q,\Z,X})\arrow[d]\\
\H^{\bullet}(X,f^* \cB^{\bullet}_{p,q,\Z,Y})\arrow[r]
& \H^{\bullet}(X,\cB^{\bullet}_{p,q,\Z,X}).
\end{tikzcd}
$$
The vertical arrows are the canonical maps and the horizontal maps are given by pull-back of smooth forms.
Notice that the composition of 
$$\H^{\bullet}(Y,\cB^{\bullet}_{p,q,\Z,Y}) \cong
\H^{\bullet}(Y,f_*f^*\cB^{\bullet}_{p,q,\Z,Y}) \to \H^{\bullet}(Y,f_* \cB^{\bullet}_{p,q,\Z,X})
\to  \H^{\bullet}(X,\cB^{\bullet}_{p,q,\Z,X})$$
is exactly the pull-back of cohomology classes.
A comparison of this diagram with the diagram given before the lemma
concludes the proof.
\end{proof}

For complex Bott-Chern cohomology, the following formula in Proposition 1 is valid, since the cohomology class can be represented by global smooth forms and since the push forward of global forms under the projection is just the integration over the second component, which commutes with the restriction on the corresponding (smooth) submanifold. 

To prove the case of integral coefficients, we need a relative version of pull back and push forward for cohomology classes. To do this, we recall some definitions of derived categories. For a more complete description, we refer to \cite{KS}. We start with the definition of a relative soft sheaf.
\begin{mydef}
Let $f:X \to Y$ be a continuous proper morphism between topological spaces and $F$ be a sheaf of abelian groups on $X$. Then we say that $F$ is $f$-soft if for any $y \in Y$, $F|_{f^{-1}(y)}$ is soft. 
\end{mydef}
In general, to define $Rf_*$ (or some right derived functor), one can take any $f_*$-injective resolution (or any relative injective resolution). 
In particular, we do not need to take an injective resolution (which is the key point of Axiom B (2)).
We verify that a $f$-soft resolution gives a $f_*$-injective resolution.
\begin{mydef} {\rm (Definition 1.8.2 in \cite{KS})}
Let $F: \mathcal{C} \to \mathcal{C}' $ be an additive functor between abelian categories. 
A full additive subcategory $\mathcal{S}$ of $\mathcal{C}$ is called injective with respect to $F$ if
\begin{enumerate}
\item
for any $X \in Ob(\mathcal{C})$ there exists $X' \in Ob(\mathcal{S})$ and an exact sequence $0 \to X \to X'$.
\item
For any exact sequence $0 \to X' \to X \to X'' \to 0$ in $\mathcal{C}$, if $X', X \in Ob(\mathcal{S})$ then $X'' \in Ob(\mathcal{S})$.
\item For any exact sequence $0 \to X' \to X \to X'' \to 0$ in $\mathcal{C}$, if $X',X,X'' \in Ob(\mathcal{S})$ then we have exact sequence
$$0 \to F(X') \to F(X) \to F(X'') \to 0.$$   
\end{enumerate}
\end{mydef}
\begin{mylem}
The subcategory formed by $f$-soft modules in $C(\Sh(X))$ is injective with respect to $f_*$ for $f$ proper.
\end{mylem}
\begin{proof}
It is a variant version of Proposition 2.5.10 in \cite{KS}. We give the proof in the relative case. 

Since any soft module is $f$-soft by definition and the subcategory formed by soft modules has enough injective element i.e. it satisfies condition 1, the subcategory formed by $f$-soft modules in $C(\Sh(X))$ also satisfies condition 1.
Notice that since $f$ is proper, for any $y \in Y$, $f^{-1}(y)$ is compact hence closed.

Condition 2 is a direct consequence of equivalence (i) and (ii) in Exercice II.19 (b) in \cite{KS}.  It says that
for any exact sequence of $\Z_X$ modules $0 \to F' \to F \to F'' \to 0$ with $F,F'$ $f$-soft and for any $y \in Y$, the hypothesis that $0 \to F'|_{f^{-1}(y)} \to F|_{f^{-1}(y)} \to F''|_{f^{-1}(y)} \to 0$ is exact implies that $F''|_{f^{-1}(y)}$ is soft. In particular, $F''$ is $f$-soft.

Now, we prove condition 3, i.e.\ that if $0 \to F' \to F \to F'' \to 0$ is an exact sequence of $f$-soft module, then there is an exact sequence
$$0 \to f_*F' \to f_* F \to f_* F'' \to 0.$$
Let $y \in Y$, we want to check that for any $s'' \in \Gamma(f^{-1}(y), F'')$ there exists $s \in \Gamma(f^{-1}(y), F)$ whose image is $s''$.
Notice that since $f$ is proper the functors $f_*$ and $f_!$ are the same. By the base change theorem (proposition 2.5.2 in \cite{KS}), we have
$$(f_*F)_y \cong \Gamma(f^{-1}(y), F|_{f^{-1}(y)}).$$
Let $K_i$ be a finite covering of $f^{-1}(y)$ by compact subsets such that there exists $s_i \in \Gamma(K_i, F)$ whose image is $s''|_{K_i}$.
This is possible from the assumption that $F \in F''$ is surjective and the fact that $f^{-1}(y)$ is compact.
Let us argue by the induction on the index of the covering to adjust the $s_i$'s such that $s_i$'s glue to a global section. For $n \geq 2$, on $(\bigcup_{i \leq n-1}K_i) \cap K_n$, we have $s'_1$ the glued section constructed by induction and $s_2 \in \Gamma(K_n, F)$.
Hence $s'_1 -s_2 \in \Gamma((\bigcup_{i \leq n-1} K_i)\cap K_n, F')$ which extends to $s' \in \Gamma(f^{-1}(y), F')$ since $F'$ is $f$-soft.
Replacing $s_2$ by $s_2+s'$ we may assume that $$s'_1|_{(\bigcup_{i \leq n-1}K_i) \cap K_n}=s_2|_{(\bigcup_{i \leq n-1}K_i) \cap K_n}.$$
Therefore after finite times induction, there exists $s \in \Gamma(f^{-1}(y),F)$ such that $s|_{K_i}=s_i$.

(Notice that condition 2 can be deduced from condition 3 by the following commutative diagram. Let $K$ be a closed subset of $f^{-1}(y)$. We have
\[\begin{tikzcd}
\Gamma(f^{-1}(y), F) \arrow{r} \arrow{d}  & \Gamma(f^{-1}(y),F'') \arrow{d}{} \\
\Gamma(K, F) \arrow{r}& \Gamma(K,F'').
\end{tikzcd}
\]
The fact that the bottom and left arrow are surjective implies that the right arrow is surjective.)
\end{proof}
We also need the following lemma (Lemma 3.1.2) in \cite{KS}.
\begin{mylem}
Let $f: X \to Y$ be a continuous map of locally compact spaces and $K$ be a flat and $f$-soft $\Z_X$ module. 
For any sheaf $G$ on $X$, $G \otimes_{\Z_X} K$ is $f$-soft.
\end{mylem}
This lemma entails the following useful corollary.
\begin{mycor} {\it
Let $X,Z$ be two complex manifolds with $Z$ compact. Let $F^{\bullet}$ be a flat complex (of sheaves of abelian groups) over $X$ and $G ^{\bullet}$ be a soft and flat complex over $Z$. Then $F^{\bullet} \boxtimes G^{\bullet}$ is flat and $q$-soft with respect to $q: X \times Z \to X$.}
\end{mycor}
\begin{proof}
The flatness part is from the fact that for abelian groups flatness is equivalent to torsion-freeness. For any $x \in X$ we have $F^{\bullet} \boxtimes G^{\bullet}|_{\{x\} \times Z}=F_x^{\bullet} \otimes_{\Z_Z} G^{\bullet}$ which, by the lemma, is $q$-soft.
\end{proof}
Now, we are prepared for the proof of Axiom B (2).
\begin{myprop} {\it
Consider the following commutative diagram where $p,q$ are the projections on the first factors
$$  \xymatrix{
    Y \times Z \ar@{^{(}->}[r]^{i_{Y \times Z}} \ar[d]_p & X \times Z \ar[d]^q \\
    Y \ar@{^{(}->}[r]^{i_Y} & X
  }$$
Assume $Z$ to be compact. Then one has in integral Bott-Chern cohomology
an equality $ i_{Y}^*q_*=p_*i_{Y \times Z}^*$.}
\end{myprop}
\begin{proof}
The idea is to use a resolution on $X \times Z$ formed by pulling back a resolution involving smooth forms on $X$, and tensoring with the pull-back of a resolution involving currents on $Z$. This gives a $q$-soft resolution, and an explicit method to calculate $Rq_*$, via Corollary 2.

Let $\mathcal{U}$ be an open covering of $X$ formed by geodesic balls with a small enough radius such that any finite intersection of such balls is diffeomorphic to a euclidean ball. 
Consider the {\v C}ech complex $\check{\mathcal{C}}^{\bullet}(\mathcal{U},\Z_X)$ associated to $\Z_X$ (a sheafified version with the same notation used in Hartshorne \cite{Har}).
More precisely,
denote $\Z_{X, i_0 \cdots i_p}$ the restriction of $\Z_X$ to $U_{i_0 \cdots i_p}=U_{i_0} \cap \cdots \cap U_{i_p}$. There exists a complex $\check{\mathcal{C}}^{\bullet}(\mathcal{U},\Z_X)$ of $\Z_X-$modules with
$\check{\mathcal{C}}^{p}(\mathcal{U},\Z_X)=\prod_{ i_0 \cdots i_p}(j_{i_0 \cdots i_p} )_{*} \Z_{X, i_0 \cdots i_p}$ and the usual {\v C}ech differential where $j_{i_0 \cdots i_p}$ is the inclusion of $U_{i_0 \cdots i_p}$ (cf. 20.34 \cite{St}).
Therefore, the {\v C}ech complex gives a resolution of $\Z_X$ (i.e. $\Z_X \to \check{\mathcal{C}}^{p}(\mathcal{U},\Z_X)$ is a quasi-isomorphism). 
By Leray theorem, $H^\bullet(X,\Z_X)$ can be represented by global {\v C}ech cocycles.
It is a flat complex on $X$ since all terms are torsion-free. Also,
$\cI_Z^{\bullet}$ is a flat and soft resolution of $\Z_Z$ on $Z$.

By Corollary 2, $\check{\mathcal{C}}^{\bullet}(\mathcal{U},\Z_X) \boxtimes \cI_Z^{\bullet}$ is a $q$-soft resolution of $\Z_{X \times Z}=\Z_X \boxtimes \Z_Z$ on $X \times Z$.  

Now we perform a similar construction for the sheaves of smooth forms.
The sheaves of smooth forms $C_{\infty, X \times Z}^{\bullet,\bullet}$ on $X \times Z$ can be viewed as flat $\Z_X-$modules and $\Z_Z-$modules.
Thus we have quasi-isomorphisms
$$C_{\infty, X \times Z}^{\bullet,\bullet} \cong C_{\infty, X \times Z}^{\bullet,\bullet} \otimes_{\Z_X} \check{\mathcal{C}}^{\bullet}(\mathcal{U},\Z_X) \cong C_{\infty, X \times Z}^{\bullet,\bullet} \otimes_{\Z_X}^L \Z_X.$$
Similarly we have quasi-isomorphisms
$$C_{\infty, X \times Z}^{\bullet,\bullet} \cong C_{\infty, X \times Z}^{\bullet,\bullet} \otimes_{\Z_Z}\cI_Z^{\bullet} \cong C_{\infty, X \times Z}^{\bullet,\bullet} \otimes_{\Z_Z}^L \Z_Z.$$
Therefore, the integral Bott-Chern complex on $X \times Z$ in the derived category is quasi-isomorphic to
\begin{equation}
\cB_{\Z, X \times Z}^{\bullet} \cong \rC(\check{\mathcal{C}}^{\bullet}(\mathcal{U},\Z_X) \boxtimes \cI_Z^{\bullet} \to \sigma_{p,\bullet}C_{\infty}^{\bullet,\bullet} \oplus \sigma_{\bullet,q} C_{\infty}^{\bullet,\bullet})[-1]
\end{equation}
with the natural inclusion morphism
where $  \sigma_{p,\bullet} C_{\infty}^{\bullet,\bullet}$ means in fact $ \check{\mathcal{C}}^{\bullet}(\mathcal{U},\Z_X) \otimes_{\Z_X} \sigma_{p,\bullet} C_{\infty, X \times Z}^{\bullet,\bullet}\otimes_{\Z_Z}\cI_Z^{\bullet}$ (similarly for $  \sigma_{\bullet,q} C_{\infty}^{\bullet,\bullet}$). 
In the following, we denote the right-handed-side complex as $\mathcal{M}^\bullet_{X \times Z}$ in (3).
Note that we have quasi-isomorphisms
$\cB_{\Z, X \times Z}^\bullet \cong \mathcal{M}^\bullet_{X \times Z} \cong \check{\mathcal{C}}^{\bullet}(\mathcal{U},\Z_X) \otimes_{\Z_X} \tilde{\cB}^\bullet_{X \times Z}$.
By K\"unneth formula (or by fineness of $C_{\infty, X \times Z}^{\bullet,\bullet}\otimes_{\Z_Z}\cI_Z^{\bullet}$), all higher cohomologies of sheaves in $\mathcal{M}^\bullet_{X \times Z}$ (before tensoring $\check{\mathcal{C}}^{\bullet}(\mathcal{U},\Z_X)$) over $U_{i_0 \cdots i_p} \times Z$  are trivial.
By Leray theorem, the hypercohomologies of $\mathcal{M}^\bullet_{X \times Z}$ can be represented by global sections.
As before, we denote $\check{\mathcal{C}}^{\bullet}(\mathcal{U},\Z_X) \otimes_{\Z_X} \tilde{\cB}_{ X \times Z}^{\bullet} \cong \tilde{\cB}_{ X \times Z}^{\bullet}$ the Bott-Chern complex involving locally integral currents and currents as in (2).
The tensor product is induced from the tensor product of sheaves of $\Z_X-$ (resp. $\Z_Z-$)modules.
Notice that the sheaves of smooth forms on $X \times Z$ are also $q$-soft.
In particular, we have
$$
Rq_* (\cB_{\Z, X \times Z}^{\bullet}) \cong q_*(\rC(\check{\mathcal{C}}^{\bullet}(\mathcal{U},\Z_X) \boxtimes \cI_Z^{\bullet} \to \sigma_{p,\bullet}C_{\infty}^{\bullet,\bullet} \oplus \sigma_{\bullet,q} C_{\infty}^{\bullet,\bullet})[-1]).
$$
We have natural morphisms $q_* \pr_2^* \cI_Z^{\bullet} \to q_* \cI_{X \times Z}^{\bullet} \to \cI_X^{\bullet}$.
It induces a morphism $q_*(\check{\mathcal{C}}^{\bullet}(\mathcal{U},\Z_X) \boxtimes \cI_Z^{\bullet}) \to (\check{\mathcal{C}}^{\bullet}(\mathcal{U},\Z_X) \otimes_{\Z_X} \cI_X^{\bullet})[-2n]$ where $n=\mathrm{dim}_{\C} Z$. 
On the other hand, since $q$ is a proper submersion, we have a canonical morphism $q_*(C^{\bullet,\bullet}_{\infty}) \to (\check{\mathcal{C}}^{\bullet}(\mathcal{U},\Z_X) \otimes_{\Z_X} \cD'^{\bullet,\bullet})[-2n]$ induced by push forward of currents.
Thus we get a morphism\vskip6pt
$q_*(\rC(\check{\mathcal{C}}^{\bullet}(\mathcal{U},\Z_X) \boxtimes \cI_Z^{\bullet} \to \sigma_{p,\bullet}C_{\infty}^{\bullet,\bullet} \oplus \sigma_{\bullet,q} C_{\infty}^{\bullet,\bullet})[-1]) \to \tilde{\mathcal{M}}^\bullet_X[-2n]:= \rC((\check{\mathcal{C}}^{\bullet}(\mathcal{U},\Z_X) \otimes_{\Z_X} \cI_X^{\bullet})\to \check{\mathcal{C}}^{\bullet}(\mathcal{U},\Z_X) \otimes_{\Z_X} \sigma_{p,\bullet}\cD'^{\bullet,\bullet} \oplus \check{\mathcal{C}}^{\bullet}(\mathcal{U},\Z_X) \otimes_{\Z_X} \sigma_{\bullet,q} \cD'^{\bullet,\bullet})[-2n-1]$.
\vskip6pt\noindent
Passing to hypercohomology, this morphism induces the push forward of integral Bott-Chern cohomology by $q$.
The push forward of cohomology classes defined in this way coincides with the previous one induced by inclusion into currents. 

Since this resolution is flat, we can also use it to define the pull-back of cohomology classes. More precisely, one can define the pull-back of the cohomology class as follows.
Since $i_{Y \times Z}=(i_Y,\id_Z)$, one has
$i_{Y \times Z}^*: \rC(\check{\mathcal{C}}^{\bullet}(\mathcal{U},\Z_X) \boxtimes \cI_Z^{\bullet} \to \sigma_{p,\bullet}C_{\infty}^{\bullet,\bullet} \oplus \sigma_{\bullet,q} C_{\infty}^{\bullet,\bullet})[-1] \to \rC(\check{\mathcal{C}}^{\bullet}(\mathcal{U}\cap Y,\Z_Y) \boxtimes \cI_Z^{\bullet}\to \sigma_{p,\bullet}C_{\infty}^{\bullet,\bullet} \oplus \sigma_{\bullet,q} C_{\infty}^{\bullet,\bullet})[-1]$
induced by pulling back forms and pulling back currents.
Here $\id_Z$ is a submersion, so the pull back of currents is well defined (and is in fact the identity!).
Passing to hypercohomology, we get another way of defining $i_{Y \times Z}^*$ for integral Bott-Chern cohomology. We next check that these two definitions coincide. The inclusion $\Z_Z \to \cI_X^{\bullet}$ induces a commutative diagram
\[\begin{tikzcd}
\Z_{X \times Z}=\Z_X \boxtimes \Z_Z \arrow{r}{i_{Y \times Z}^*} \arrow{d}  &\Z_{Y \times Z}=\Z_Y \boxtimes \Z_Z \arrow{d}{} \\
\check{\mathcal{C}}^{\bullet}(\mathcal{U}, \Z_X) \boxtimes \cI_Z^{\bullet} \arrow{r}{i_{Y \times Z}^*}& \check{\mathcal{C}}^{\bullet}(\mathcal{U} \cap Y, \Z_Y) \boxtimes \cI_Z^{\bullet}.
\end{tikzcd}
\]
This commutative diagram implies that the two definitions of pull back coincide.

Similar arguments show that the push forward by $p_*$ can be defined using the corresponding resolutions.
It is non trivial to study $i_Y^* q_*$ since the image of $q_*$ is valued in $H^\bullet(X, \tilde{\cB}_X^\bullet)$ 
where $i_Y^*$ is not always well defined for currents on $X$.
Note that the image of $q_*$ is valued in $H^\bullet(X,\Im( q_*(\mathcal{M}_{X \times Z}^\bullet) \to \tilde{\mathcal{M}}_X^\bullet[-2n]))$ where the sections can be locally written as
$\int f(x,z) d \mu(z) $ where $f(x,z)$ are smooth forms on $X \times Z$ and $d \mu(z)$ locally integral currents (hence with measures coefficients).
Note that pull-back of currents by $i_Y^*$ is well defined for such currents valued in $\Im( p_*(\mathcal{M}_{Y \times Z}^\bullet) \to \tilde{\mathcal{M}}_Y^\bullet[-2n])$.
$\Im( q_*(\mathcal{M}_{X \times Z}^\bullet) \to \tilde{\mathcal{M}}_X^\bullet[-2n])$ is quasi-isomorphic to $\cB_X^\bullet$ since 
by the proof of Dolbeault-Grothendieck lemma (cf. Chap. I Lemma 3.29 \cite{agbook}), a solution of $\dbar u=v$ for any local $\dbar-$closed current $v$ valued in $\Im( q_*(\mathcal{M}_{X \times Z}^\bullet) \to \tilde{\mathcal{M}}_X^\bullet[-2n])$ could be constructed explicitly by integration formula which is also valued in $\Im( q_*(\mathcal{M}_{X \times Z}^\bullet) \to \tilde{\mathcal{M}}_X^\bullet[-2n])$.
More precisely, Dolbeault-Grothendieck lemma implies that for fixed $k$, the complex
$\Im( q_*(C_{\infty, X \times Z}^{\bullet,\bullet} \otimes_{\Z_Z}\cI_Z^{\bullet})\to \cD'^{\bullet,\bullet}_X[-2n] \to \cD'^{k+2n,\bullet}_X[-2n]) $ formed by currents and forms of fixed first index as a subcomplex of $ \cD'^{k,\bullet}_X$ is quasi isomorphic to $\Omega^k_X$.
On the other hand, $q_* (\check{\mathcal{C}}^{\bullet}(\mathcal{U},\Z_X) \boxtimes \cI_Z^{\bullet}) \to \check{\mathcal{C}}^{\bullet}(\mathcal{U},\Z_X) \otimes \cI_X^{\bullet}[-2n] $ factors through $\check{\mathcal{C}}^{\bullet}(\mathcal{U},\Z_X)$ with surjective map onto $\check{\mathcal{C}}^{\bullet}(\mathcal{U},\Z_X)$.
In particular, $\Im (q_* (\check{\mathcal{C}}^{\bullet}(\mathcal{U},\Z_X) \boxtimes \cI_Z^{\bullet}) \to \check{\mathcal{C}}^{\bullet}(\mathcal{U},\Z_X) \otimes \cI_X^{\bullet}[-2n] )$ is quasi isomorphic to $\check{\mathcal{C}}^{\bullet}(\mathcal{U},\Z_X)$.

Consider the following commutative diagram
$$
\begin{tikzcd}
{H^\bullet(X, \cB_X^\bullet)} \arrow[d, "\cong"] \arrow[r]                                     & {H^\bullet(Y, i_Y^* \cB_X^\bullet)} \arrow[d] \arrow[r]                                              & {H^\bullet(Y, \cB_Y^\bullet)} \arrow[d, "\cong"] \\
{H^\bullet(X, \Im( q_*(\mathcal{M}_{X \times Z}^\bullet) \to \tilde{\mathcal{M}}_X^\bullet[-2n]))} \arrow[r] & {H^\bullet(Y, i_Y^* \Im( q_*(\mathcal{M}_{X \times Z}^\bullet) \to \tilde{\mathcal{M}}_X^\bullet[-2n]))} \arrow[r] & {H^\bullet(Y, \tilde{\mathcal{M}}_Y^\bullet)}.           
\end{tikzcd}
$$
Thus the pull back of currents valued in $\Im( q_*(\mathcal{M}_{X \times Z}^\bullet) \to \tilde{\mathcal{M}}_X^\bullet[-2n])$ induces the pull back of cohomology classes by $i_Y^*$.

Since the resolution is relatively soft with respect to $p$ or $q$ and flat, the hypercohomologies in the morphisms
$H^{\bullet,\bullet}(X \times Z, \mathcal{M}_{ X \times Z}) \to H^{\bullet,\bullet}(Y \times Z, \mathcal{M}_{ Y \times Z}) \to H^{\bullet-n,\bullet-n}(Y, \tilde{\mathcal{M}}_{Y})$
(resp. $H^{\bullet,\bullet}(X \times Z, \mathcal{M}_{ X \times Z}) \to H^{\bullet-n,\bullet-n}(X,  \tilde{\mathcal{M}}_{X}) \to H^{\bullet-n,\bullet-n}(Y, \tilde{\mathcal{M}}_{Y})$)
can be represented by global sections which define the same pull-back and push-forward in Bott-Chern cohomologies as defined in the beginning of this section.
The sections are formed by currents valued in $\Im( q_*(\mathcal{M}_{X \times Z}^\bullet) \to \tilde{\mathcal{M}}_X^\bullet[-2n])$ (or valued in $\Im( p_*(\mathcal{M}_{Y \times Z}^\bullet) \to \tilde{\mathcal{M}}_Y^\bullet[-2n])$) and forms on the open set of $U \times Z$ or $(U \cap Y) \times Z$ for some open set $U$ of $X$, which is some intersection of the open sets in the cover~$\mathcal{U}$. The equality asserted in the proposition is satisfied for such forms and currents. 
In particular, 
$i_Y^* \int f(x,z) d \mu(z)= \int i_Y^* f(x,z) d \mu(z)  $ for  $f(x,z)$ smooth forms on $U \times Z$ and $d \mu(z)$ locally integral currents on $Z$.
This concludes the proof.
\end{proof}
Axiom B (2) can be generalised to any proper morphism $f: Y \to X$, not necessarily a closed immersion with $Z$ compact, which would be needed to prove Axiom B (3).

By the fact that the pull-back of a current is always well defined in the case of submersion, one gets the following proposition.
\begin{myprop} {\it
Let $Y,Z$ be complex manifolds with $Z$ compact and natural projection $p: Y \times Z \to Y$.
The pull-back of global representatives in terms of currents induces the pull-back of integral Bott-Chern cohomologies defined above in terms of forms. 
}
\end{myprop}
\begin{proof}
For any connected open set $V \subset Y$, we have the following commutative diagram
$$  \xymatrix{
    Z_Y(V) \ar[r]^{p^*} \ar[d] & Z_{Y \times Z}(p^{-1}(V)) \ar[d] \\
    \cI_Y^{\bullet}(V) \ar[r]^{p^*} & \cI_{Y \times Z}^{\bullet}(p^{-1}(V))
  }.$$
The vertical arrow is given by associating the constant 1 to the integral current associated with $[V]$ (resp. $[p^{-1}(V)]$). 

Passing to hypercohomology, inclusion of forms and constants $\cB^\bullet$ into currents and locally integral currents $\tilde{\cB}^\bullet$ induces isomorphism on hypercohomology, so the morphisms of integral Bott-Chern cohomology groups induced by pulling back forms and pulling back currents are the same. In other words the commutative diagram
\[ \begin{tikzcd}
p^*\cB^{\bullet}_Y \arrow{r}{} \arrow[swap]{d}{} & p^*\tilde{\cB}^{\bullet}_Y \arrow{d}{} \\%
\cB^{\bullet}_{Y \times Z} \arrow{r}{}& \tilde{\cB}^{\bullet}_{Y \times Z}
\end{tikzcd}
\]
induces in hypercohomology the commutative diagram
\[ \begin{tikzcd}
H^{*}(Y, \cB^{\bullet}_Y) \arrow{r}{\cong} \arrow[swap]{d}{} & H^{*}(Y , \tilde{\cB}^{\bullet}_Y) \arrow{d}{}
\\%
H^{*}(Y \times Z,p^* \cB^{\bullet}_Y) \arrow{r}{} \arrow[swap]{d}{} & H^{*}(Y \times Z,p^* \tilde{\cB}^{\bullet}_Y) \arrow{d}{} \\%
H^*(Y \times Z,\cB^{\bullet}_{Y \times Z}) \arrow{r}{\cong}& H^*(Y \times Z, \tilde{\cB}^{\bullet}_{Y \times Z}).
\end{tikzcd}
\] 
Here the terms containing a tilde indicate complexes involving currents, and the terms without a tilde indicate complexes involving locally constant sheaves or forms.
\end{proof}
\begin{myprop} {\it
Let $f: Y \to X$ be a proper map between compact complex manifolds. Assume that $Z$ is a compact complex manifold.
Consider the following commutative diagram where $p,q$ are the projections on the first factors
$$  \xymatrix{
    Y \times Z \ar[r]^{(f, \id)} \ar[d]_p & X \times Z \ar[d]^q \\
    Y \ar[r]^{f} & X.
  }$$
Then one has in integral Bott-Chern cohomology
an equality $ f^*q_*=p_*(f, \id)^*$.}
\end{myprop}
\begin{proof}
There are two ways to prove it.
The first way is to consider the following diagram
$$
\begin{tikzcd}
Y \times Z \arrow[d, , "p_{Y \times Z/Y}"] \arrow[r, "{(id,f,id)}", hook] & Y \times X \times Z \arrow[r, "p_{X \times Z}"] \arrow[d, "p_{Y \times X}"] & X \times Z \arrow[d, "p_{X \times Z/X}"] \\
Y \arrow[r, "{(id,f)}", hook]                       & Y \times X \arrow[r, "p_{Y \times X/X}"]                                    & X.                                       
\end{tikzcd}
$$
To show that $f^* p_{X \times Z/X,*}=p_{Y \times Z/Y,*}(f,id)^*$,
using Proposition 1, it is enough to show that
$p^*_{Y \times X/X} p_{X \times Z/X,*}=p_{Y \times X,*}p_{X \times Z}^*$.
By Proposition 2,
since all the maps in the right square are projections,
the pullback of integral Bott-Chern cohomology can be induced from the pullback of currents as well as push forward (for proper morphism such that the push forward of currents is well defined).
In particular, to show the equality at the level of cohomology, it is enough to show the equality for global currents representing the corresponding cohomology class.
Dually, it is enough to show for smooth form $\omega$ on $Y \times Z$ with compact support,
$ p_{X \times Z/X}^*p_{Y \times X/X,*} \omega= p_{X \times Z,*} p_{Y \times X}^* \omega$
which is true since we always integrate along the factor of $Y$. 

The second way is following the proof of Proposition 1.
Let $\mathcal{U}$ be a cover of $X$ with small geodesic balls.
Let $\mathcal{V}$ be a cover of $Y$ with small geodesic balls such that the image of each geodesic ball is contained in some open set in $\mathcal{U}$.
In previous proof of Proposition 1, we take $\mathcal{V}$ to be the intersection of $\mathcal{U}$ with $Y$.
The rest of the proof works identically.
\end{proof}
\begin{myrem}
{\em
In the rational coefficient case, instead of the sheaves of locally integral currents $\cI^\bullet_X$, one should consider the image of the natural map $\cI_X^\bullet \otimes_{\Z_X} \Q_X$ to the sheaves of currents which will be denoted by $\cI_{X,\Q}^{\bullet}$.
(Note that by Lemma 5, $\cI_X^\bullet \otimes_{\Z_X} \Q_X$ is soft by considering $f$ as a map to a point.)
More precisely,
in (2), one should consider 
$$\Q(p) \to \cI^{\bullet}_{X,\Q} \xrightarrow{\Delta} \sigma_p {\cD_X'}^{\bullet} \oplus \sigma_q {\cD_X'}^{\bullet}.$$
Note that the local sections of $\cI_X^\bullet \otimes_{\Z_X} \Q_X$ are the rational linear combinations of locally integral currents which can be seen as currents.
In the rational coefficient case, we work in the (derived) categories of sheaves of $\Q_X-$modules instead of abelian groups.
For example, we should replace (3) by
$$
\cB_{\Q, X \times Z}^{\bullet} \cong \rC(\check{\mathcal{C}}^{\bullet}(\mathcal{U},\Q_X) \boxtimes \cI_{Z,\Q}^{\bullet} \to \sigma_{p,\bullet}C_{\infty}^{\bullet,\bullet} \oplus \sigma_{\bullet,q} C_{\infty}^{\bullet,\bullet})[-1]$$
with
$$C_{\infty, X \times Z}^{\bullet,\bullet} \cong \check{\mathcal{C}}^{\bullet}(\mathcal{U},\Q_X) \otimes_{\Q_X} C_{\infty, X \times Z}^{\bullet,\bullet} \otimes_{\Q_Z}\cI_{Z,\Q}^{\bullet}.$$
$\cI^\bullet_{X,\Q}$ is a (flat) resolution of $\Q_X$.
In fact, for any germ of $\cI^\bullet_{X,\Q}$,  some positive multiple of this germ is locally integral.
As the germs of $\cI_X^\bullet$ is a resolution of the germs of $\Z_X$, $\cI^\bullet_{X,\Q}$ is a resolution of $\Q_X$.
The only issue is to check that $\cI^\bullet_{X,\Q}$ is soft.
The following arguments are modified from Page 57 of \cite{HK74}.
Let $Z$ be a closed subset of $X$.
Let $S \in \Gamma(Z, \cI^\bullet_{X,\Q})$ which we want to extend onto $X$.
Assume that $S$ is defined over an open neighbourhood $U$ of $Z$.
Let $f$ be a smooth function on $X$ which is identically equal to 1 on an open neighbourhood of $X \setminus U$ and is equal to 0 near $Z$.
Let $\chi_\epsilon$ be the characteristic function of the set $\{f < \epsilon \}$.
Fix some $0< \epsilon_0 <1$.
Cover $X$ by open sets $X_i$ such that $X_i \Subset X_{i+1}$.
For any $i$, since $\overline{X}_i$ is compact, there exists $N_i \in \N^*$ such that $N_i S|_{\{f \leq \epsilon_0 \} \cap \overline{X}_i}$ values in $\cI^\bullet$.
By Federer slicing theorem, for a.e. $\epsilon\leq \epsilon_0 $, $N_i \chi_\epsilon S$ is locally integral when restricting to $X_i \cap \{f < \epsilon_0 \}$.
Thus for a.e. $\epsilon \leq \epsilon_0$, $ \chi_\epsilon S \in \Gamma(X, \cI^\bullet_{X,\Q})$ which extends $S$.
Hence all our results in the integral coefficient case also work in the rational coefficient case.


Let $f: X \to Y$ be a proper morphism between Riemannian manifolds of relative diemension $r$.
By compactness of the fibers, the natural morphism $f_* \cI^\bullet_X \to \cI^\bullet_Y[-r]$ induces a natural morphism
$f_*( \cI^\bullet_X \otimes_{\Z_X} \Q_X) \to \cI^\bullet_Y \otimes_{\Z_Y} \Q_Y[-r]$ since some multiple of a local section of $f_* ( \cI^\bullet_X \otimes_{\Z_X} \Q_X)$ is integral coefficient (up to restriction to a relatively compact open subset).
Taking image in the sheaf of currents, the induced natural morphism is independent of the choice of mulitple.
In other words, we have a natural morphism $f_* \cI^\bullet_{X,\Q} \to \cI^\bullet_{Z,\Q}[-r]$.
}
\end{myrem}

\section{Multiplication of the Bott-Chern cohomology ring}
In this section, we discuss a natural ring structure of the integral Bott-Chern cohomology and we verify the projection formula (Axiom B(1)).
Some calculation of this part is borrowed from an unpublished work of Junyan Cao.
The detailed verification of all the calculations can be founded in the author's PhD thesis \cite{Wu20}.
 
The complex Bott-Chern cohomology is represented by global differential forms. The exterior product of forms induces the multiplication of cohomology classes. To define a multiplication of integral Bott-Chern cohomology which preserves the ring structure under the canonical map from the integral Bott-Chern cohomology to the complex Bott-Chern cohomology, we start by defining a modified version of multiplication of Deligne cohomology. Recall that the integral Deligne complex $\D(p)^{\bullet}$ is the complex in $C(\Sh(X))$
$$\Z(p) \to \cO \to \Omega^1 \to \cdots \to \Omega^{p-1} \to 0. $$ 
The integral Deligne complex admits a multiplication structure as follows.
$$\cup: \D(p)^{\bullet} \otimes_{\Z_X} \D(q)^{\bullet} \to \D(p+q)^{\bullet}$$
$$
x \cup y=
\begin{cases}
x \cdot y, \text{if} \, \deg{(x)}=0 \\
x \wedge dy, \text{if} \, \deg{(x)}>0 \text{ and } \deg (y)=p\\
0, \text{ otherwise.}
\end{cases}
$$
$\cup$ is a morphism of $\Z_X$-module sheaf complexes by a direct verification. A modified version of multiplication is given in the following definition.
\begin{mydef}
For the integral Deligne complex, we define
$$\cup: \D(p)^{\bullet} \otimes_{\Z_X} \D(q)^{\bullet} \to \D(p+q)^{\bullet}$$
$$
x \cup y=
\begin{cases}
x \cdot y, \text{if} \, \deg{(y)}=0 \\
(-1)^p dx \wedge y, \text{if} \, \deg{(y)}>0 \text{ and } \deg (x)=p\\
0, \text{ otherwise.}
\end{cases}
$$
\end{mydef}
We can verify that $\cup$ yields a well defined morphism of complexes, namely that
$$d(x \cup y)=dx \cup y +(-1)^{\deg(x)} x \cup dy.$$
\begin{myrem}{\rm
For the definition of multiplication in the integral Bott-Chern complex, we need a modified Deligne complex where we change the signs. To be more precise, we consider the complex
$$\Z(p) \xrightarrow{-1} \cO \to \Omega^1 \to \cdots \to \Omega^{p-1} \to 0. $$ 
In this case, we define the multiplication as follows:
$$
x \cup y=
\begin{cases}
x \cdot y, \text{if} \, \deg{(y)}=0 \\
(-1)^{p-1} dx \wedge y, \text{if} \, \deg{(y)}>0 \text{ and } \deg (x)=p\\
0, \text{ otherwise.}
\end{cases}
$$
}
\end{myrem}
\begin{myprop} {\it
The multiplication is associative and homotopy graded-commutative. Thus, it induces a structure of an anti-commutative ring with unit on the integral Deligne cohomology.}
\end{myprop}
\begin{proof}
Considering $\alpha\in \cD(p)^{\bullet}$, $\tilde{\alpha}\in \mathcal{D}(p')^{\bullet}$ and $\deg\hspace{2pt} (\alpha)=i$, $\deg\hspace{2pt} (\tilde{\alpha})=j$, 
we prove the formula
$$(-1)^{ij}\alpha\cup\tilde{\alpha}=\tilde{\alpha}\cup\alpha+ (d H+ H d)(\tilde{\alpha}\otimes\alpha).$$
Here $d$ is the differential of $\mathcal{D}(p)^{\bullet} \otimes \mathcal{D}(p')^{\bullet}$, and $d(\alpha\otimes\beta)$ is defined by $d(\alpha\otimes\beta)=d\hspace{2pt} \alpha \otimes \beta+(-1)^{\deg (\alpha)}\alpha\otimes d\hspace{2pt}\beta$. 
The modified homotopy operator $H$ is defined by: $H(\tilde{\alpha}\otimes\alpha)=(-1)^{j-1}\tilde{\alpha}\wedge\alpha$, if $i \neq 0, j\neq 0$. Otherwise, $H(\tilde{\alpha}\otimes\alpha)=0$. 
\end{proof}
\begin{myrem}{\rm 
Similarly, for the integral Bott-Chern cohomology, the modified Deligne complex admits a homotopy operator defined by: $H(\tilde{\alpha}\otimes\alpha)=(-1)^{j}\tilde{\alpha}\wedge\alpha$, if $i \neq 0, j\neq 0$. Otherwise, $H(\tilde{\alpha}\otimes\alpha)=0$. We also have the equality: 
$$(-1)^{ij}\alpha\cup\tilde{\alpha}=\tilde{\alpha}\cup\alpha+ (d H+ H d)(\tilde{\alpha}\otimes\alpha).$$
}
\end{myrem}
Once we have defined a morphism from a tensor product of two complexes to another complex. It naturally induces a product on the hypercohomology class. For self-containedness, we recall the construction.
\begin{mydef}Consider two complexes of sheaves $\mathcal{A}^{\bullet}, \mathcal{B}^{\bullet}$, such that there exists a multiplication denoted by $\cup$: $\mathcal{A}^{\bullet}\otimes_{\Z} \mathcal{B}^{\bullet}\rightarrow \mathcal{C}^{\bullet}$,
$\alpha \otimes \beta \mapsto \alpha \cup \beta$
 satisfying the relation $d(\alpha\cup \beta)=(d\alpha)\cup \beta+(-1)^{\deg (\alpha)}\alpha\cup d\beta$. 
Then one can define a product between $\mathbb{H}^{\bullet}(\mathcal{A}^{\bullet})$ and $\mathbb{H}^{\bullet}(\mathcal{B}^{\bullet})$ as follows:
let $\beta\in\check{C}^k(\mathcal{A}^{l})$ and $\tilde{\beta}\in\check{C}^{k'}(\mathcal{B}^{l'})$
(where {\v C} means {\v C}ech hypercocycle). One defines a {\v C}ech hypercocycle $\beta\cdot\tilde{\beta}\in\check{C}^{k+k'}(\mathcal{C}^{l+l'})$ by
$$(\beta\cdot\tilde{\beta})_{j_0\ldots j_{k+k'}}:=(-1)^{k\cdot l'}\beta_{j_0\ldots j_k}\cup\tilde{\beta}_{j_k\ldots j_{k+k'}}.$$
\end{mydef}
We can check the derivation relation:
$$\dcechg(\beta\cdot\tilde{\beta})=(\dcechg\beta)\cdot\tilde{\beta}+(-1)^{k+l}\beta\cdot(\dcechg\tilde{\beta})$$ 
where $\dcechg\beta=(-1)^{l}\delta\beta+d \beta$, $\delta$ is the {\v C}ech differential.

The multiplicative structure on the integral Deligne complex induces a multiplicative structure on the integral Bott-Chern complex as follows. 
We denote $\epsilon_{\D}$ the canonical morphism of complexes from the integral Bott-Chern complex $\cB_{p,q,\Z}^{\bullet}$ to the integral Deligne complex $\D(p)^{\bullet}$. We denote $\overline{\epsilon_{\D}}$ the canonical morphism of complexes from the integral Bott-Chern complex $\cB_{p,q,\Z}^{\bullet}$ to the modified conjugated integral Deligne complex $\overline{\D(q)^{\bullet}}:= 0 \to \Z(q) \xrightarrow{-1}\overline{\cO_X} \to \dots \to \overline{\Omega^{q-1}_X} \to 0$ with a multiplication of $(2\pi i)^{q-p}$ at degree 0. 
The modified multiplication of modified integral Deligne complex in Remark 2 induces a multiplication of modified conjugated integral Deligne complex. 
These two canonical maps induce a multiplicative structure on the integral Bott-Chern complex as follows. Let $y',y''$ be two elements of $\D(p)^{i}, \overline{\D(q)^{i}}$ over the same open set for some $i$. If $i=0$, there exists a unique element $x$ of $\cB_{p,q,\Z}^{0}$ such that $\epsilon_{\D}(x)=y'$ and $\overline{\epsilon_{\D}}(x)=y'$ if and only if they satisfy $y ''=(2\pi i)^{q-p} y'$. The existence of the unique element is trivial for all positive degree.
Hence we can define the multiplication $x \cup x'$ of two elements $x,x'$ of $\cB_{p,q,\Z}^{i}$ and $\cB_{p',q',\Z}^{j}$ respectively just to be the unique element such that $\epsilon_{\D}(x \cup x')=x \cup x'$ and $\overline{\epsilon_{\D}}(x \cup x')=x \cup x'$ with the cup product of Deligne complex and the modified cup product of modified conjugated Deligne complex respectively. At degree 0, the multiplication is just the multiplication of the two integer at degree 0 up to a constant satisfying the compatible condition. Therefore the multiplication of the integral Bott-Chern complex is well-defined. In conclusion, the cup product of the complex is given explicitly by the following definition.
\begin{mydef}
Let $w, \tilde{w}$ be two elements of the complex $\mathcal{B}_{p,q}^{\bullet} \otimes_{\Z} \mathcal{B}_{p',q'}^{\bullet}$,
and let us use the following diagrams to denote the elements $w, \tilde{w}$ of mixed degrees
$$w=\left(c,\begin{array}{l} u^{0,0},\ldots, u^{p-1,0}\\ v^{0,0},\ldots\ldots, v^{0,q-1}\end{array}\right),\qquad \tilde{w}=\left(\tilde{c},\begin{array}{l} \tilde{u}^{0,0},\ldots\ldots\ldots, \tilde{u}^{p'-1,0}\\ \tilde{v}^{0,0},\ldots, \tilde{v}^{0,q'-1}\end{array}\right).$$
For instance, at degree 0, we denote $w$ by $c$, at degree 1, we denote $w$ by $(u^{0,0}, v^{0,0})$ etc.
With the same notation, the cup product $w \cup \tilde{w}$ is represented by the diagram 
$$\left(c\wedge\tilde{c},\begin{array}{l} c\wedge\tilde{u}^{0,0}\,,\,\ldots\ldots\ldots,c\wedge\tilde{u}^{p'-1,0}\,,\,u^{0,0}\wedge\partial\tilde{u}^{p'-1,0}\,,\,\ldots\,,\,u^{p-1,0}\wedge\partial\tilde{u}^{p'-1,0}\\ v^{0,0}\wedge\tilde{c}\,,\,\ldots\,,\,v^{0,q-1}\wedge\tilde{c}\,,\,(-1)^{q-1}\dbar v^{0,q-1}\wedge\tilde{v}^{0,0}\,,\,\ldots\ldots\,,\,(-1)^{q-1}\dbar v^{0,q-1}\wedge\tilde{v}^{0,q'-1}\end{array}\right).$$
\end{mydef}
The cup product of integral Bott-Chern cohomology is given explicitly by the following diagram.
\begin{mydef}
Let $w, \tilde{w}$ be two representatives of hypercocycles of the complex $\mathcal{B}_{p,q}^{\bullet} \otimes_{\Z} \mathcal{B}_{p',q'}^{\bullet}$, and let
us use the following diagrams to denote the elements $w, \tilde{w}$
$$w=\left(c,\begin{array}{l} u^{0,0},\ldots, u^{p-1,0}\\ v^{0,0},\ldots\ldots, v^{0,q-1}\end{array}\right),\qquad \tilde{w}=\left(\tilde{c},\begin{array}{l} \tilde{u}^{0,0},\ldots\ldots\ldots, \tilde{u}^{p'-1,0}\\ \tilde{v}^{0,0},\ldots, \tilde{v}^{0,q'-1}\end{array}\right).$$
For instance, at degree 0, we denote by $c$ an element in $\check{C}^{p+q}(\cB_{p,q}^0)$, at degree 1, we denote by $(u^{0,0}, v^{0,0})$ an element in $\check{C}^{p+q-1}(\cB_{p,q}^1)$ etc.
With the same notation, the cup product $w \cup \tilde{w}$ is represented by the diagram 
$$\left(c\wedge\tilde{c},\begin{array}{l} \epsilon^{0,*} c\wedge\tilde{u}^{0,0}\,,\,\ldots\ldots\ldots,\epsilon^{p'-1,*} c\wedge\tilde{u}^{p'-1,0}\,,\, \epsilon^{p',*} u^{0,0}\wedge\partial\tilde{u}^{p'-1,0}\,,\,\ldots\,,\,\epsilon^{p+p'-1,*}u^{p-1,0}\wedge\partial\tilde{u}^{p'-1,0}\\ \epsilon^{*,0}v^{0,0}\wedge\tilde{c}\,,\,\ldots\,,\,\epsilon^{*,q-1}v^{0,q-1}\wedge\tilde{c}\,,\,\epsilon^{*,q}\dbar v^{0,q-1}\wedge\tilde{v}^{0,0}\,,\,\ldots\ldots\,,\,\epsilon^{*,q+q'-1}\dbar v^{0,q-1}\wedge\tilde{v}^{0,q'-1}\end{array}\right).$$
The signs $\epsilon^{R,*}$,$\epsilon^{*,S}$ are given as follows:
$$
\epsilon^{R,*}=
\begin{cases}
(-1)^{(p+q)(R+1)}, \text{if} \, R \leq p'-1 \\
(-1)^{p'(R+p+q)}, \text{if} \, R \geq p'
\end{cases}
$$
$$
\epsilon^{*,S}=
\begin{cases}
1, \text{if} \, S \leq q-1 \\
(-1)^{pS+(p+1)(q+1)}, \text{if} \, S \geq q
\end{cases}.
$$
\end{mydef}
Notice that this cup product is just the cup product defined in \cite{Sch}. 
Let us also notice that there exists a more obvious natural product structure on the complex Bott-Chern cohomology induced by the wedge product of forms. The signs in the cup product defined in \cite{Sch} are exactly taken in such a way that the two products coincide under the natural morphism.
The natural inclusion of the integral Bott-Chern complex into the complex Bott-Chern complex induces a ring morphism in hypercohomology. The morphism of complexes $\epsilon_{\cD}$ also induces a ring morphism in hypercohomology.
\begin{myprop} {\it
The multiplication is anti-commutative. Thus, it induces a structure of an anti-commutative ring with unit on the integral Bott-Chern cohomology.}
\end{myprop}
\begin{proof}
As for Deligne cohomology, there is a natural homotopy operator. We identify the degree 0 sheaf in the integral Bott-Chern class $\Z(p)$ with a subsheaf of $\Z(p) \oplus \Z(q)$ via the map $1 \mapsto (1,(2\pi i)^{q-p})$. In this way, we can include the integral Bott-Chern complex into the direct sum of the integral Deligne cohomology and the (modified) conjugate integral Deligne complex.
We define $H: \cB_{p,q,\Z}^{\bullet} \otimes_{\Z_X} \cB_{p',q',\Z}^{\bullet} \to \cB_{p+p',q+q',\Z}^{\bullet}$ by the formula for any element $\varphi^i=(a^i,b^i) \in \cB_{p,q,\Z}^{i}$, $\psi^j=(a'^j,b'^j) \in \cB_{p',q',\Z}^{j}$, 
$$
H(\varphi^i \otimes \psi^j):=
\begin{cases}
((-1)^i a^i \wedge b^j,(-1)^j a'^i \wedge b'^j ), \text{if} \, i \neq 0,j \neq 0 \\
0, \text{ otherwise.}
\end{cases}
$$
This is well defined since at degree 0, the homotopy operator is 0 map. 
We have checked that
$$(-1)^{ij} \psi^j \cup \varphi^i=\varphi^i \cup \psi^j+(dH+Hd)(\varphi^i \otimes \psi^j).$$
Therefore, passing to hypercohomology, we have defined an anti-commutative ring structure on the integral Bott-Chern cohomology.
For reference, the formulas for the homotopy operator of the integral Deligne complex can be found in \cite{EV}.
\end{proof}
We write $\varphi \cdot\psi$ for the multiplication of cohomology classes. There exists also another description of cup product following \cite{EV} by introducing a modified version of the Deligne-Beilinson complex. In this way,
the projection formula can be expressed more formally.
It differs
in sign from the original definitions in \cite{EV} in order to get ring morphisms.

The advantage of the Deligne-Beilinson complex is that the multiplication is either 0 or weight product of two forms. When changing the complex involving forms by the complex involving currents, it becomes clearer what the sign should be.

The modified Deligne complex is quasi isomorphic to the following modified Deligne-Beilinson complex
$$A(p)^{\bullet}=\rC(\Z(p) \oplus F^p\Omega_X^{\bullet} \xrightarrow{-\epsilon -\ell} \Omega_X^{\bullet})[-1]$$
where $\epsilon,\ell$ are the natural maps.
A quasi-isomorphism $\alpha:\cD(p)^{\bullet} \to A(p)^{\bullet}$ can be given by
\[ \begin{tikzcd}
\Z(p) \arrow{r}{} \arrow{d}{\alpha_0} 
& \cO_X \arrow{d}{\alpha_1}  \arrow{r}{} 
&\cdots \arrow{r}{} 
&\Omega^{p-2} \arrow{d}{\alpha_{p-1}} \arrow{r}{}
&\Omega^{p-1} \arrow{d}{\alpha_{p}} \arrow{r}{}
&0 \arrow{d}{}
\\%
\Z(p) \arrow{r}{-\epsilon}
& \cO_X  \arrow{r}{-\delta_1} 
&\cdots \arrow{r}{} 
& \Omega^{p-2} \arrow{r}{-\delta_{p-1}} 
& \Omega^p \oplus \Omega^{p-1} \arrow{r}{-\delta_{p}} 
& \Omega^{p+1} \oplus \Omega^{p} \cdots
\end{tikzcd}
\]
with $\alpha_p(\omega)=(-1)^p(d\omega,\omega)$ and $\alpha_i(\omega)=(-1)^i \omega$. The symbol $\delta$ denotes the differential of the mapping cone, where in particular
$$\delta_{p-1}(\eta)=(0,d\eta), \delta_{p}(\psi, \eta)=(-d\psi,-\psi+d\eta).$$
The mapping cone has a negative sign, by the convention that for a complex $(A^{\bullet}, d^{\bullet})$, the
complex $A^{\bullet}[d]$ has a differential in degree $n$ defined by $(-1)^d d^{n-d}$. 

We define a modified cup product $\cup_0$ by the following table which is a morphism of complexes.
$$\begin{array}{c|c|c|c|}   & a_q & f_q & \omega_q \\ \hline a_p & a_p \cdot a_q & 0 & 0 
\\ \hline f_p & 0  & -f_p \wedge f_q & (-1)^{\deg(f_p)-1}f_p \wedge \omega_q 
\\ \hline \omega_p & \omega_p \cdot a_q  & 0 & 0
\end{array}~.$$
We can verify that the map from modified Deligne complex to the modified Delinge-Beilinson complex is also commutative under the modified cup product. 

Hence passing to hypercohomology, we have a ring isomorphism for the modified Deligne cohomology and the modified Deligne-Beilinson cohomology.

As above, the cup product of the Deligne-Beilinson complex and the modified cup product of the conjugate modified Deligne-Beilinson complex induce a cup product on the integral Bott-Chern complex. Indeed, the latter is quasi-isomorphic to
$$\cB_{p,q,\Z}^{\bullet}=\rC(\Z(p) \oplus F^p\Omega_X^{\bullet} \oplus F^q\overline{\Omega_X^{\bullet}} \xrightarrow{(\epsilon,-(2\pi i)^{q-p}\epsilon)+(-\ell, -\overline{\ell})} \Omega_X^{\bullet} \oplus \overline{\Omega_X^{\bullet}})[-1]$$
where $\epsilon$ is the natural map $\Z(p) \to \Omega_X^{\bullet}$ and $\ell, \overline{\ell}$ are the natural maps $F^p\Omega_X^{\bullet} \to \Omega_X^{\bullet}$ and $F^p\overline{\Omega_X^{\bullet}} \to \overline{\Omega_X^{\bullet}}$.
With this quasi-isomorphism it becomes easier to check the projection formula.
\begin{myprop}(Projection formula) {\it 
For any holomorphic morphism $g$,
one has
$$(1)\, g^*{\varphi}\cdot g^*{\psi}=g^*{(\varphi\cdot\psi)}.$$
For a proper morphism $f$, one has
$$(2) \,f_*(\varphi\cdot f^*\psi)=f_*\varphi\cdot\psi.$$}
\end{myprop}
\begin{proof}
For the first equality, we can in fact check that on the level of complexes
$$g^*{\varphi}\cup g^*{\psi}=g^*{(\varphi \cup \psi)}. $$
Below, we concentrate ourselves on the proof of the second equality.
The integral Bott-Chern complex is quasi-isomorphic to the complex
$$\tilde{\cB}_{p,q,\Z}^{\bullet}:=\rC(\cI_X^{\bullet} \oplus s(F^{p,\bullet}\cD_X^{'\bullet,\bullet}) \oplus s(F^{\bullet,q}\cD_X^{'\bullet,\bullet}) \xrightarrow{(\epsilon,-\overline{\epsilon})+(-\ell, -\overline{\ell})} \cD_X^{'\bullet} \oplus \cD_X^{'\bullet})[-1]$$
where $\epsilon$ is the natural map $\cI_X^{\bullet} \to \cD_X^{'\bullet}$, $s(F^{p,\bullet}\cD_X^{'\bullet,\bullet})$ is the total complex of $F^{p,\bullet}\cD_X^{'\bullet,\bullet}$, i.e.\ the direct sum of spaces of currents of bidegree $(k,l)$ $(k \leq p)$, and $\ell, \overline{\ell}$ are the natural maps $s(F^{p,\bullet}\cD_X^{'\bullet,\bullet}) \to \cD_X^{'\bullet}$ and $s(F^{\bullet,q}\cD_X^{'\bullet,\bullet}) \to \cD_X^{'\bullet}$.
We start by defining a multiplication between $\cB_{p',q',\Z}^{\bullet}$ and $\tilde{\cB}_{p',q',\Z}^{\bullet}$ that is compatible with the multiplication of the integral Bott-Chern complex. In this way, we avoid the problematic  weight product of two currents. We first perform a similar construction for the integral Deligne complex. One can represent the product
$$\cup_0:A(p)^{\bullet} \otimes \rC(\cI_X^{\bullet} \oplus s(F^{q,\bullet}\cD_X^{'\bullet,\bullet})  \xrightarrow{\epsilon-\ell} \cD_X^{'\bullet})[-1] \to
\rC(\cI_X^{\bullet} \oplus s(F^{p+q,\bullet}\cD_X^{'\bullet,\bullet})  \xrightarrow{\epsilon-\ell} \cD_X^{'\bullet})[-1]$$
by the following table
$$\begin{array}{c|c|c|c|}   & a_q & f_q & \omega_q \\ \hline a_p & a_p \cdot a_q & 0 & a_p \cdot \omega_q 
\\ \hline f_p & 0  & f_p \wedge f_q & 0
\\ \hline \omega_p & 0  & \omega_p \wedge f_q & 0
\end{array} $$
representing elements of
$$\begin{array}{c|c|c|c|}   & \cI_X^{\bullet} & s(F^{q,\bullet}\cD_X^{'\bullet,\bullet})  & \cD_X^{'\bullet} \\ \hline \Z(p) & \cI_X^{\bullet} & 0 & \cD_X^{'\bullet} 
\\ \hline F^p\Omega^{\bullet} & 0  & s(F^{p+q,\bullet}\cD_X^{'\bullet,\bullet}) & 0
\\ \hline \Omega^{\bullet} &  0 & \cD_X^{'\bullet} & 0
\end{array}~. $$
Notice that the wedge product of smooth forms and currents is always well-defined. We also observe that since a locally integral current is represented by a generalised measure by the Riesz representation theorem, it defines a current of order 0.
We can check that the multiplication is a morphism of complexes. 

One can change the definition of $\cup_0$ for the modified Deligne complex by introducing a different sign for the morphism at degree 0, according to the table
$$\begin{array}{c|c|c|c|}   & a_q & f_q & \omega_q \\ \hline a_p & a_p \cdot a_q & 0 & 0 
\\ \hline f_p & 0  & -f_p \wedge f_q & (-1)^{\deg(f_p)-1}f_p \wedge \omega_q 
\\ \hline \omega_p & \omega_p \cdot a_q  & 0 & 0
\end{array} $$
representing elements of
$$\begin{array}{c|c|c|c|}   & \cI_X^{\bullet} & s(F^{q,\bullet}\cD_X^{'\bullet,\bullet})  & \cD_X^{'\bullet} \\ \hline \Z(p) & \cI_X^{\bullet} & 0 & 0
\\ \hline F^p\Omega^{\bullet} & 0  & s(F^{p+q,\bullet}\cD_X^{'\bullet,\bullet}) & \cD_X^{'\bullet}
\\ \hline \Omega^{\bullet} &  \cD_X^{'\bullet}  & 0 & 0
\end{array}~.$$

We have the following commutative diagram of $\Z_X$-modules, where, as before, the multiplication of Deligne complex and the modified multiplication of the modified Deligne complex induce the multiplication of the integral Bott-Chern complex
\[ \begin{tikzcd}
\cB(p,q,\Z)^{\bullet} \otimes_{\Z_X} \cB(p',q',\Z)^{\bullet}\arrow{r}{\cup} \arrow[swap]{d}{} & \cB(p+p',q+q',\Z)^{\bullet} \arrow{d}{} \\%
\cB(p,q,\Z)^{\bullet} \otimes_{\Z_X} \tilde{\cB}(p',q',\Z)^{\bullet} \arrow{r}{\cup_0}& \tilde{\cB}(p+p',q+q',\Z)^{\bullet}.
\end{tikzcd}
\]
The vertical arrow is induced by the morphism of complexes $\Z(p) \to \cI^{\bullet}_X$. The ``gluing condition'' used to define the multiplication of the integral Bott-Chern complex, starting from the Deligne complex and the conjugate (modified) Deligne complex, is that $\Z(p) \otimes \cI_X^{\bullet} \to \cI_X^{\bullet}$ should be the same for both complexes. Now, the second equality comes from the straightforward check 
$$f_*(f^*\psi \cup_0 \varphi)=\psi \cup_0 f_*\varphi.$$
This equality induces as follows the desired formula on the level of hypercohomology. 
By the algebraic K\"unneth formula (cf. Theorem 15.5 in \cite{agbook}), we have a morphism 
$$H^*(Ra_{Y*}\cB(p,q,\Z)) \otimes H^*(Ra_{Y*}Rf_* \tilde{\cB}(p',q',\Z)) \to H^*(Ra_{Y*}\cB(p,q,\Z) \otimes^L Ra_{Y*}(Rf_* \tilde{\cB}(p',q',\Z))).$$
Notice that since $\Z$ is a PID, $\cB(p',q',\Z)$, $\tilde{\cB}(p,q,\Z), a_{X*} \tilde{\cB}(p',q',\Z),a_{Y*} \tilde{\cB}(p,q,\Z)$ are torsion free and flat. Notice also that $\tilde{\cB}(p,q,\Z)$ is also a soft complex. There is in fact no need to write functors $R$ and $L$ in the above morphism. 
More precisely, $Ra_{Y*}\cB(p,q,\Z) \otimes^L Ra_{Y*}(Rf_* \tilde{\cB}(p,q,\Z))=Ra_{Y*}\cB(p,q,\Z) \otimes a_{Y*}( f_* \tilde{\cB}(p,q,\Z))=a_{Y*}\tilde{\cB}(p,q,\Z) \otimes a_{X*} \tilde{\cB}(p',q',\Z)$.
We have proven that the following diagram commutes:
$$
\begin{tikzcd}
\cB(p,q,\Z) \otimes f_* \tilde{\cB}(p',q',\Z)
\arrow[d]
\arrow[r] 
 & f_*(f^*\cB(p,q,\Z) \otimes  \tilde{\cB}(p',q',\Z))\arrow[d]\\
\cB(p,q,\Z) \otimes  \tilde{\cB}(p',q',\Z)[-2d]  \arrow[r] 
& \tilde{\cB}(p+p',q+q',\Z))[-2d] 
\end{tikzcd}(*)
$$
where $d$ is the relative dimension of $f$.
More precisely, the right arrow is given by
$ f_*(f^*\cB(p,q,\Z) \otimes  \tilde{\cB}(p',q',\Z)) \to  f_*(  \tilde{\cB}(p+p',q+q',\Z)) \to \tilde{\cB}(p+p',q+q',\Z))[-2d] $.
Let us observe that a tensor product of a soft and flat complex by any complex is soft by Lemma 5. 
In fact, we can write functors $R$ and $L$ in the above morphism since all complexes are soft and flat. 
By taking $Ra_{Y*}$
we have the following induced map
$$
\begin{tikzcd}{}\kern-5pt
Ra_{Y*}\cB(p,q,\Z){\otimes^L}Ra_{X*} \tilde{\cB}(p',q',\Z) 
\to\arrow[rd]\kern-32pt&
Ra_{Y*}(\cB(p,q,\Z){\otimes^L}Rf_* \tilde{\cB}(p',q',\Z))
\arrow[d]
\\
& Ra_{Y*}\tilde{\cB}(p+p',q+q',\Z)[-2d].
\end{tikzcd}
$$
(Remark that the symbol $f^*$ used here is denoted $f^{-1}$ by some authors.)
The left arrow is the natural morphism and the left-down arrow is just the composition.
Taking hypercohomology and composing with the morphism in the K\"unneth formula give the projection formula.

The order for taking the cup product is unimportant when passing to hypercohomology, since the integral Bott-Chern cohomology is anti-commutative. This finishes the proof of the projection formula.
\end{proof}
\section{Chern classes of a vector bundle}
In this part we give a construction of the Chern class of a vector bundle in the integral Bott-Chern cohomology. It is borrowed from Junyan Cao (personal communication). The general line is Grothendieck's construction of Chern classes of a vector bundle via the splitting principle.
In particular, we prove Axiom A stated in the introduction.

We first recall the definition of the first Chern class of a line bundle in integral Bott-Chern cohomology, following \cite{Sch}.

Let $L$ be a holomorphic line bundle over $X$ and $\mathcal{U}=(U_j)$ be an open covering of $X$ with connected intersections such that on each $U_j$, $L$ is locally trivial by a nowhere-vanishing section $e_j$. We denote $g_{jk}$ the transition function defined on $U_j\cap U_k$ defined by the relation $e_k(x)=g_{jk}(x)e_j(x)$. Perhaps with further refinement of the open covering, we can suppose that $g_{jk}=\exp(u_{jk})$. The element
$$\{g_{jk}\}\in\check{H}^1(\Ub,\cO^*)\cong H^1(X,\cO^{*})$$
determines the isomorphic class of $L$.
Let $h$ be a hermitian metric on $L$ and we denote by $D$ the Chern connection associated with $(L,h)$ and by $\Theta$ the curvature of the Chern connection.
On $U_j$, the Chern connection is given by the formula
$$D(\xi_j(x)e_j(x))=(d\xi_j(x)-\d\varphi_j(x)\xi_j(x))\otimes e_j(x)$$
where $\varphi_j$ is the local weight function of the metric under the trivialisation defined by
$$e^{-\varphi_j(z)}=|e_j(z)|_h^2,$$ 
which verifies the compatibility condition on $U_j\cap U_k$ :
$$-\varphi_k+\varphi_j=u_{jk}+\overline{u_{jk}}.$$ 
We define the {\v C}ech 2-cocycle $\delta(u_{jk})$ to be $(2\pi i c_{jkl})$ which means on $U_{jkl}$
$$2 \pi i c_{jkl}=u_{jk}-u_{jl}+u_{kl}.$$
Taking exponential map on the both sides we know
$$\exp{2 \pi i c_{jkl}}=g_{jk}*g_{jl}^{-1}*g_{kl}=1$$
which in particular shows $(2\pi i c_{jkl})\in \check{C}^2(X, \Z(1))$ a 2$-\check{C}$ech cocycle with value in $\Z(1)$. 
We define the first Chern class of $L$ in the integral Bott-Chern cohomology to be
$$c_1(L)_{BC,\Z}:=\left\{(2\pi i c_{jkl}),(u_{jk}),(\overline{u_{jk}})\right\}\in H^{1,1}_{BC}(X,\Z).$$
By possible further refinement of the cover, we can assume that two line bundles $L_1, L_2$ are trivialised at the same time.
Since the transition functions of $L_1 \otimes L_2$ are the product of corresponding ones of $L_1, L_2$,
it is easy to see that $c_1$ is a group morphism by the above construction.

We prove in what follows that
this hypercocycle also represents the Chern class of $L$ in the complex Bott-Chern cohomology.
For the complex Bott-Chern cohomology, the corresponding global representative (1,1)-form via the quasi-isomorphic complex $\lc^{\bullet}_{p,q}[1]$ which is defined with $p=1,q=1$
$$\lc^k_{p,q}=\bigoplus_{\substack{r+s=k \\ r<p,s<q}}\ec^{r,s}\quad \mathrm{if }k\leq p+q-2,$$
$$\lc^{k-1}_{p-1,q-1}=\bigoplus_{\substack{r+s=k \\ r\geq p,s\geq q}}\ec^{r,s}\quad \mathrm{if }k\geq p+q,$$
with differential
$$\lc^0\xrightarrow{\pr_{\lc^1}\circ d}\lc^1\xrightarrow{\pr_{\lc^2}\circ d}\ldots\to\lc^{k-2}\xrightarrow{\frac{i}{2\pi}\ddbar}\lc^{k-1}
\xrightarrow{d}\lc^k\xrightarrow{d}\ldots$$
is just the global form with $\frac{i}{2\pi}\ddbar \varphi_j$ on $U_j$. 
Notice that the complex $\lc^{\bullet}_{p,q}$ is acyclic.
The proof of the quasi isomorphism between $\lc^{\bullet}_{p,q}$ and $\cB_{p,q}^{\bullet}$ can be found in section 12 Chap VI of \cite{agbook}.
(Notice that in \cite{Sch}, the operator $\frac{i}{2 \pi} \ddbar$ is changed by $\ddbar$. Here we take this choice so that the first Chern class of a line bundle in the integral Bott-Chern class has image as the first Chern class in the complex Bott-Chern class under the canonical morphism.)

With the same notation as in \cite{Sch}, $\alpha^{0,0}$ can be chosen to be $(\varphi_j)$, so the global representative is $\theta^{0,0}=\frac{i}{2\pi}\ddbar \alpha^{0,0}$. This is exactly the curvature form on $U_j$. 
Therefore the hypercocycle of $\cB_{1,1,\Z}^{\bullet}$ viewed as a hypercocycle of $\cB_{1,1,\C}^{\bullet}$ corresponding to $ \Theta$ is
$$\{\Theta\}\longleftrightarrow \left\{(2\pi i c_{jkl}),(u_{jk}),(\overline{u_{jk}})\right\}.$$
Observe that the first Chern class of the complex Bott-Chern cohomology is just represented by the curvature. We denote by $\epsilon_{BC}$ the canonical map from the integral Bott-Chern complex to the complex Bott-Chern complex. We have in hypercohomology
$$\epsilon_{BC}\,c_1(L)_{BC,\Z}=c_1(L)_{BC}.$$
Notice that the Chern classes of a vector bundle in integral Bott-Chern cohomology (which will be defined below) and in complex Bott-Chern cohomology are both defined by means of the splitting principle, in such a way that for any $d$ and any vector bundle $E$ we have
$$\epsilon_{BC}\,c_d(E)_{BC,\Z}=c_d(E)_{BC}.$$
To construct the Chern class of a vector bundle, we use Grothendieck's splitting principle. We begin by proving a Leray-Hirsch type theorem for the integral Bott-Chern cohomology. This theorem is a direct consequence of the Hodge decomposition theorem and of the Leray-Hirsch theorem for De Rham cohomology, in case $X$ is a compact K\"ahler manifold. Here we give a generalisation to arbitrary compact complex manifolds.
Before giving the statement in the integral Bott-Chern cohomology, we need a K\"unneth type theorem of the same nature for Dolbeault cohomology, and which will be used in a further induction process. 
\begin{mythm}(Theorem 1.1 \cite{Meng20}, Proposition 11 \cite{ASTT20})

Let $X$ be a compact complex manifold and $E$ be a vector bundle of rank $r$ on $X$. One has an isomorphism
$$\bigoplus\limits_{s\leq r-1} H^{p-s, k-p-s}(X)\cdot c_{1}^{s}(\mathcal{O}(1)) \rightarrow H^{p,k-p}(\mathbb{P}(E)).$$
\end{mythm}

Now, we prove the principal proposition of this section, namely a Leray-Hirsch type theorem for the integral Bott-Chern cohomology.
\begin{myprop} {\it
Let $X$ be a compact complex manifold, $E$ a vector bundle of rank $r$ over it. Then, we have
 $$\mathbb{H}^{k}(\mathbb{P}(E), \mathcal{B}_{p,q,\Z}^{\bullet})=\mathbb{H}^{k}(X, \mathcal{B}_{p,q,\Z}^{\bullet})\oplus\mathbb{H}^{k-2}(X, \mathcal{B}_{p-1,q-1,\Z}^{\bullet})\cdot\omega\oplus\cdots\oplus\mathbb{H}^{k-2r+2}(X, \mathcal{B}_{p-r+1,q-r+1,\Z}^{\bullet})\cdot\omega^{r-1}$$
 where $\omega$ is the first Chern class of the tautological line bundle over $\mathbb{P}(E)$ in $H^{1,1} _{BC}(\mathbb{P}(E), \mathbb{Z})$ as defined above.}
\end{myprop}
The isomorphism is induced by pull back followed by the cup product of the integral Bott-Chern cohomology defined in Section 3 which is a morphism of complexes.
Note that since we do not assume that $p,q \geq r-1 $, some notations need to be explained for the negative index case.
In the proposition we use the following notations. 
\\
If $p<0$ (resp. $q<0$), we denote $\mathcal{B}_{p,q,\Z}^{\bullet}=\mathcal{B}_{0,q,\Z}^{\bullet}$ (resp. $\mathcal{B}_{p,0,\Z}^{\bullet}$). The morphism  $$F:\bigoplus\limits_{s\leq r-1} \mathbb{H}^{k-2s}(X, \mathcal{B}_{p-s, q-s,\Z}^{\bullet})\cdot \omega^{s} \rightarrow \mathbb{H}^{k}(\mathbb{P}(E),\mathcal{B}_{p,q}) $$ is defined as follows: let $\pi: \mathbb{P}(E)\rightarrow X$;

If $s\leq \min{(p,q)}$, $F(\alpha\cdot\omega^{s})=\pi^{\ast}(\alpha)\cdot \omega^{s}$,

If $s\geq p$, $F(\alpha\cdot\omega^{s})=\pi^{\ast}(\alpha)\cdot \omega^{p}\cdot \pr_{0,1}(\omega)^{s-p}$,

If $s\geq q$, $F(\alpha\cdot\omega^{s})=\pi^{\ast}(\alpha)\cdot \omega^{q}\cdot \pr_{1,0}(\omega)^{s-q}$,

where the projection $\pr_{0,1}$ is induced by the canonical projection from $\mathcal{B}_{1,1,\Z}^{\bullet}$ to $\mathcal{B}_{0,1,\Z}^{\bullet}$.
Similarly $\pr_{1,0}$ is induced by the projection to $\mathcal{B}_{1,0,\Z}^{\bullet}$.

Notice that when $p=q=r$, $k=2r$, this is just the normal splitting principle without the complicated notations.
\begin{proof}
The idea is to use the exact sequence
$$0\rightarrow\Omega^{p}[p]\rightarrow \mathcal{B}_{p+1, q,\Z}^{\bullet}\rightarrow\mathcal{B}_{p,q,\Z}^{\bullet}\rightarrow 0$$
to reduce the proof to the Dolbeault case.
In this proof, we use the usual convention for differential forms that for $p <0$, $\Omega^p[p]=0$.
We begin by proving that the following diagram is commutative and that its two lines are exact:
 \tiny{
\[
\begin{tikzcd}
\bigoplus\limits_{s\leq r-1}\kern-5pt \mathbb{H}^{k-2s}(X, \Omega^{p-s}[p-s])\cdot \omega^{s} 
 \to\kern-25pt{} \arrow{d}{}  
  & \bigoplus\limits_{s\leq r-1}\kern-5pt\mathbb{H}^{k-2s}(X, \mathcal{B}_{p+1-s, q-s,\Z}^{\bullet})\cdot \omega^{s}
   \to\kern-25pt{} \arrow{d}{} 
   & \bigoplus\limits_{s\leq r-1}\kern-5pt\mathbb{H}^{k-2s}(X, \mathcal{B}_{p-s, q-s,\Z}^{\bullet})\cdot \omega^{s} 
    \to\kern-25pt{} \arrow{d}{} 
     & \bigoplus\limits_{s\leq r-1}\kern-5pt\mathbb{H}^{k-2s+1}(X, \Omega^{p-s}[p-s])\cdot \omega^{s}  \arrow{d}{} \\
 \mathbb{H}^{k}(\mathbb{P}(E),\Omega^{p}[p])
 \arrow{r}{} & \mathbb{H}^{k}(\mathbb{P}(E),\mathcal{B}_{p+1,q,\Z}^{\bullet})
 \arrow{r}{} & \mathbb{H}^{k}(\mathbb{P}(E),\mathcal{B}_{p,q,\Z}^{\bullet})
 \arrow{r}{} &  \mathbb{H}^{k+1}(\mathbb{P}(E),\Omega^{p}[p])
\end{tikzcd}
 \]}
 \normalsize{We first check the exactness of the two lines.}
The exactness is just obtained from the long exact sequence associated with the short exact sequence of sheaves. We now check the commutativity of the first square.
$$\xymatrix{\mathbb{H}^{k-2s}(X, \Omega^{p-s}[p-s])\cdot \omega^{s} \ar[d]^-{G}\ar[r]^-{i}  & \mathbb{H}^{k-2s}(X, \mathcal{B}_{p+1-s, q-s,\Z}^{\bullet})\cdot \omega^{s} \ar[d]^-{F}\\ \mathbb{H}^{k}(\mathbb{P}(E),\Omega^{p}[p])\ar[r]^-{i} & \mathbb{H}^{k}(\mathbb{P}(E),\mathcal{B}_{p+1,q,\Z}^{\bullet}).}$$
The morphism $G$ is induced from the following morphism of complexes $\Omega^{p-s}[p-s] \otimes_{\Z_X} \cB_{1,1,\Z}^{\bullet} \to \Omega^{p-s+1}[p-s+1]$. Denote the germs as $\alpha\in \Omega^{p-s}[p-s]$, $\omega=\left(\widetilde{c},\beta; \overline{\beta}\right)$. We define
$$G(\alpha \otimes  \beta)=\alpha\wedge (\partial \beta).$$ 
We take it equal to zero otherwise.
$G$ defines a morphism at the level of hypercohomology.
From now on, we do not pay attention to write $\alpha$ or $\pi^* \alpha$ when the context should make the meaning clear.
To prove the commutativity at the level of hypercohomology, it is enough to show the commutativity at the level of complexes. It is enough to check the commutativity for the case $s \leq p$. We have
$$i(\alpha\wedge (\partial \beta)^{s})=(0,0,0....\alpha^{p-s}\wedge (\partial \beta)^{s};0 ),
$$
which is equal to the image of $F \circ i$.
\\
We check the commutativity of the second square. Let $\alpha=\left(c,\alpha_{0},...,\alpha_{p-s}; \overline{\beta}_{0},..., \overline{\beta}_{q-s-1}\right)$,  $\omega=\left(\widetilde{c},\beta; \overline{\beta}\right)$ be the representatives of hypercocycles. 
If $s\leq p$, the horizontal morphism just consists of forgetting the term involving $\alpha_{p-s}$, thus it is commutative.
Otherwise, $\alpha=\left(c, \overline{\beta}_{0},..., \overline{\beta}_{q-s-1}\right)$ and the morphism is induced by the identity map at the level of complexes, so it is commutative.
\\
We check the commutativity of the third square.
$$\xymatrix{
\bigoplus \mathbb{H}^{k-2s}(X, \mathcal{B}_{p-s, q-s,\Z}^{\bullet})\cdot \omega^{s} \ar[d]^-{F}\ar[r]^-{i}  & \bigoplus \mathbb{H}^{k-2s+1}(X, \Omega^{p-s}[p-s])\cdot \omega^{s} \ar[d]^-{G} \\
\mathbb{H}^{k}(\mathbb{P}(E),\mathcal{B}_{p,q,\Z}^{\bullet})\ar[r]^-{i} &  \mathbb{H}^{k+1}(\mathbb{P}(E),\Omega^{p}[p]).} $$ 
If $s\leq p-1$, take a representative of hypercocycle $\alpha=\left(c,\alpha_{0},...,\alpha_{p-s-1}; \overline{\beta}_{0},..., \overline{\beta}_{q-s-1}\right)$, which is the image of hypercocycle of $\mathcal{B}_{p-s+1, q-s,\Z}^{\bullet}$ $\left(c,\alpha_{0},...,\alpha_{p-s-1},0; \overline{\beta}_{0},..., \overline{\beta}_{q-s-1}\right)$.
By the definition of the connecting morphism,
$i(\alpha)$ can be taken as the degree $(p-s)$ element of the hypercocycle $\check{\delta} \left(c,\alpha_{0},...,\alpha_{p-s-1},0; \overline{\beta}_{0},..., \overline{\beta}_{q-s-1}\right)$ which is $\partial\alpha_{p-s-1}$.
Hence
$$G(i(\alpha))=\partial\alpha_{p-s-1}\wedge (\partial \beta)^{s}.
$$ 
On the other hand,
 $i(F(\alpha))=\partial(\alpha_{p-s-1}\wedge(\partial\beta)^{s})=\partial\alpha_{p-s-1}\wedge(\partial\beta)^{s}.$

If $s=p$, we take a representative of the hypercocycle $\alpha=\left(c, \overline{\beta}_{0},..., \overline{\beta}_{q-s-1}\right)$, which is the image of the hypercocycle $\left(c,0; \overline{\beta}_{0},..., \overline{\beta}_{q-s-1}\right)$ of $\mathcal{B}_{1, q-s,\Z}^{\bullet}$. By definition of the connecting morphism,
$i(\alpha)$ can be taken as the degree 0 element of the hypercocycle $\check{\delta} \left(c,0; \overline{\beta}_{0},..., \overline{\beta}_{q-s-1}\right)$, which is $c$.
\\
Therefore $i(\alpha)=c$ and $G(i(\alpha))=c\wedge(\partial\beta)^{s}$. The two elements with highest degrees in the hypercocycle $F(\alpha)$ are $c\wedge \beta \wedge (\partial \beta)^{s-1}$ and $c\wedge  (\partial \beta)^{s}$. Now, $i(F (\alpha))$ is the degree $p$ element of the hypercocycle $\check{\delta}(F( \alpha))$, namely
$$i(F(\alpha))=\partial(c\wedge\beta\wedge(\partial\beta)^{s-1})=G(i(\alpha)).$$  
If $s<p$, the sequence 
$$0\rightarrow\Omega^{p-s}[p]\rightarrow \mathcal{B}_{p+1-s, q,\Z}^{\bullet}\xrightarrow{\sim} \mathcal{B}_{p-s,q,\Z}^{\bullet}\rightarrow 0$$
is an isomorphism between the second and third terms, which therefore induces a zero connecting morphism. The diagram is also commutative in this case.

At this point, all the asserted commutativity properties have been checked.

Using the five lemma to perform an induction on $p$, we have to prove that the following morphism is an isomorphism:
$$G: \bigoplus\limits_{s\leq r-1} \mathbb{H}^{k-2s}(X, \Omega^{p-s}[p-s])\cdot \omega^{s} \rightarrow \mathbb{H}^{k}(\mathbb{P}(E),\Omega^{p}[p]).$$
On the {\v C}ech cohomology groups $\check{H}^{p}(X,\Omega^{q})$, one can introduce a ring structure by the wedge product
$$\check{H}^{p}(X,\Omega^{q})\times\check{H}^{p'}(X,\Omega^{q'})\rightarrow \check{H}^{p+p'}(X,\Omega^{q+q'}).$$
On the other hand, using the De Rham-Weil isomorphism, we have a canonical isomorphism 
$$\phi:\check{H}^{p}(X,\Omega^{q})\rightarrow H^{q,p}(X, \mathbb{C}).$$
The isomorphism is compatible with the ring structure of Dolbeault cohomology, possibly up to a sign (sketched in \cite{Suw09} and detailed in \cite{Wu20}). 

Now we prove that $G$ is an isomorphism.  Let $\omega=(c, \beta; \overline{\beta})$, so that by definition $G(\alpha \cdot \omega^s)$ is represented by the $k$-hypercocycle $G(\alpha\cdot \omega^{s})=\pi^{*}(\alpha)\wedge (\partial \beta)^{s}$.
By the construction of the Chern class of the line bundle $\cO(1)$, we have $ \beta_{jk} + \overline{\beta_{jk}}=\phi_{j}-\phi_{k}$
which implies
$$\partial \beta_{jk}=\partial (\phi_{j}-\phi_{k}).$$
A diagram chasing procedure similar to the proof of the De Rham-Weil isomorphism gives that the image of $\partial \beta_{jk}$ in $H^{1,1}(\mathbb{P}(E), \mathbb{C})$ is $-\dbar(\partial \phi_j)$, where the later form is the curvature. 
The negative sign comes from the convention that if we denote $\delta, d$ the differentials of a double complex, $d \delta+\delta d=0$.
Therefore, to define a double complex from the {\v C}ech complex and $\dbar$-complex, we have to add a negative sign following the parity.
In conclusion $\omega$ represents $c_{1}(\mathcal{O}(1))$, hence by the Leray-Hirsch type theorem for Dolbeault cohomology, the isomorphism $G$ is settled.

To conclude the proof of the proposition, the five lemma and an induction on $p$ reduce the proof to the case $p=0$. It is enough to show that
$$\mathbb{H}^{k}(\mathbb{P}(E), \mathcal{B}_{0,q,\Z}^{\bullet})=\mathbb{H}^{k}(X, \mathcal{B}_{0,q, \Z}^{\bullet})\oplus\mathbb{H}^{k-2}(X, \mathcal{B}_{0,q-1,\Z}^{\bullet})\cdot\omega\oplus\cdots\oplus\mathbb{H}^{k-2r+2}(X, \mathcal{B}_{0,q-r+1,\Z}^{\bullet})\cdot\omega^{r-1}.$$
The short exact sequence 
$0\rightarrow \overline{\Omega}^{q}[q]\rightarrow \mathcal{B}_{0, q+1,\Z}^{\bullet}\rightarrow \mathcal{B}_{0,q,\Z}^{\bullet}\rightarrow 0$ induces the two lines of the following diagram are exact.
\tiny{$$\xymatrix{
\bigoplus\limits_{s\leq r-1} \mathbb{H}^{k-2s}(X, \overline{\Omega}^{q-s}[q-s])\cdot \omega^{s} \ar[d]\to\kern-30pt & \bigoplus\limits_{s\leq r-1} \mathbb{H}^{k-2s}(X, \mathcal{B}_{0, q+1-s})\cdot \omega^{s} \ar[d]\to\kern-30pt & \bigoplus\limits_{s\leq r-1} \mathbb{H}^{k-2s}(X, \mathcal{B}_{0, q-s})\cdot \omega^{s} \ar[d]\to\kern-30pt  & \bigoplus\limits_{s\leq r-1} \mathbb{H}^{k-2s+1}(X, \overline{\Omega^{q-s}}[q-s])\cdot \omega^{s} \ar[d] \\
 \mathbb{H}^{k}(\mathbb{P}(E),\overline{\Omega}^{q}[q])\ar[r] & \mathbb{H}^{k}(\mathbb{P}(E),\mathcal{B}_{0,q+1})\ar[r] & \mathbb{H}^{k}(\mathbb{P}(E),\mathcal{B}_{0,q})\ar[r] &  \mathbb{H}^{k+1}(\mathbb{P}(E),\overline{\Omega}^{q}[q])
} ,$$}%
\normalsize{Here we change the connecting morphism of the first line with a sign $(-1)^s$ on the relevant terms.} 
This change does not affect the exactness of sequence but ensures the commutativity of the diagram. 
As before, we check that the diagram is commutative. 
To simply the sign in the cup product of Bott-Chern cohomology, we use the anti-commutativity of the integral Bott-Chern class. For any class $\alpha$, $\alpha \cdot \omega= \omega \cdot \alpha$.
Notice that since $p=0$, $\omega$ is in fact $\pr_{0,1} \omega$.
With the same notations as before, 
this time the morphism $G$ is induced by the morphism of complexes $\cB_{1,1,\Z}^{\bullet} \otimes_{\Z_X}  \overline{\Omega^{q-s}}[p-s]\to \overline{\Omega^{q-s+1}}[p-s+1]$. Denote the germs by $\alpha\in \overline{\Omega^{q-s}}[p-s]$ and $\omega=\left(\widetilde{c},\beta; \overline{\beta}\right)$. We define
$$G(\overline{\beta} \otimes \alpha )=(\dbar \overline{\beta})\wedge\alpha$$ 
and take it equal to zero otherwise.
To check the commutativity of the first square, it is enough to check the commutativity at the level of complexes for the case $s \leq q$.
$$i((\dbar \overline{\beta})^{s}\wedge\alpha )=(0,0;0....(\dbar \overline{\beta})^{s}\wedge \alpha^{q-s})
$$
which is equal to the image of $F \circ i$. The commutativity of the second square is easy. 
\\
We now check the commutativity of the third square.
Take hypercocycles $\alpha=(c, \overline{v}_{0},...,\overline{v}_{q-s})$, $\pr_{0,1}(\omega)=(\tilde{c},\overline{\beta})$. 
It is enough to check the case $s \leq q$, otherwise the connecting morphism is zero map. 
If $s \leq q-1$, 
the image of $\alpha$ under the connecting morphism is $\dbar \overline{v}_{q-s}$. The image at the lower right corner of the diagram is $(-\dbar \overline{\beta})^{s}\wedge \dbar \overline{v}_{q-s}$. 
(The sign comes from the change of the signs in the first line. )
On the other hand, the image under the connecting morphism of $F(\alpha)=(\dbar \overline{\beta})^{s}\wedge  \overline{v}_{q-s}$  is $\dbar((\dbar \overline{\beta})^{s}\wedge  \overline{v}_{q-s})=(-\dbar \overline{\beta})^{s}\wedge \dbar \overline{v}_{q-s}$.
\\
If $s = q$, 
the image of $\alpha$ under the connecting morphism is $-c$. The image at the lower right corner of the diagram is $(-\dbar \overline{\beta})^{s}\wedge -c$.  On the other hand, the elements with the two highest degrees in the hypercocycle $F(\alpha)$ are $(\dbar \overline{\beta})^{s-1} \wedge \overline{\beta} \wedge c$ and $(\dbar \overline{\beta})^{s} \wedge  c$. The image of the first one under the connecting morphism is $\dbar((\dbar \overline{\beta})^{s-1} \wedge \overline{\beta} \wedge c)=(-\dbar \overline{\beta})^{s}\wedge -c$.

By the five lemma, similar arguments as those given above reduce the induction on $q$ to the case $q=0$, $p=0$. In the case $\cB_{p,q,\Z}^{\bullet}=\Z$, the isomorphism is trivial.
\end{proof}
The splitting principle can thus be applied and gives the following definition of Chern classes for a vector bundle.
\begin{mydef}
Taking $p=q=r$, $k=2r$, there are unique elements $c_{i}\in H^{i,i} _{BC}(X, \mathbb{Z})$, such that
 $$\omega^{r}+\sum{(-1)^{i}\pi^{*}(c_{i})\cdot\omega^{r-i}}=0$$
where $\omega=c_{1}(\mathcal{O}(1))$ by the above Proposition 7. We define the Chern classes of a vector bundle $E$ in the integral Bott-Chern cohomology to be precisely the $c_{i}$. 
\end{mydef}

We now prove some elementary properties of Chern classes in the integral Bott-Chern cohomology.
In particular, we check that Axiom A of the introduction holds. 
Let us first observe that such Chern classes are unique, since they satisfy the Grothendieck axioms for Chern classes included in Axiom A.
Notice that the first Chern class defines a group morphism $c_1: \Pic(X) \to A^1(X)$ which is compatible with pull-backs by construction.

The first property is the Whitney formula.
\begin{myprop} {\it
Let $0\rightarrow E\rightarrow F\rightarrow G\rightarrow 0$ be a short exact sequence of holomorphic vector bundles. Then we have $\ch(E)+\ch(G)=\ch(F)$ and $c(E) \cdot c(G)=c(F)$.}
\end{myprop}
\begin{proof}
On $X\times \mathbb{P}^{1}$, there exists a short exact sequence of holomorphic vector bundles
$$0\rightarrow\tilde{E}\rightarrow\tilde{F}\rightarrow\tilde{G}\rightarrow 0,$$ 
such that the restriction of exact sequence on the complex submanifold $X\times \{0\}$ is $0\rightarrow E\rightarrow F\rightarrow G\rightarrow 0$
and the restriction on $X\times \{\infty\}$ is $0\rightarrow E\rightarrow E\oplus G\rightarrow G\rightarrow 0.$
The existence of such a sequence can be found for example on Page 80 in \cite{Sou}.
In the case of a direct sum, we obviously have the formulas $\ch(G)+\ch(E)=\ch(E\oplus G)$ and $c(E \oplus G)=c(E) \cdot c(G)$ by the splitting principle. 
More precisely, by the splitting principle, we can reduce to the case of the direct sum of line bundles.
It is enough to show that for any vector bundle $E$ the total Serge class satisfies
$s(E \oplus \cO_X)=s(E)$
by Example 3.1.1 of \cite{Fu}.
This formula is shown in the following Lemma 7.

On the other hand, we have the following commutative diagram for every point $a\in \mathbb{P}^{1}$: 
$$\xymatrix{X\times \mathbb{P}^{1} \ar[r]^-{\pi}  & X \\ X\ar[u]^-{i_{a}}\ar[ur]^{\mathrm{Id}} & } .$$
The identity element of the ring $\oplus_{k,p,q} \mathbb{H}^{k}(X, \cB_{p,q,\Z}^{\bullet})$ is the element in $\mathbb{H}^{0}(X, \cB_{0,0,\Z}^{\bullet})$ represented by the constant $1 \in \Z(0)(X)$ (more precisely the 0-cocycle $1 \in \Z(0)(U_i)$ for each $U_i$ in the open covering).
Via the quasi-isomorphism, it can also be represented by the integral current associated with $X$.  We denote this element by $\mathrm{Id}_{X}$.
By the projection formula we have for every $\alpha\in \mathbb{H}^{\bullet}(X\times\mathbb{P}^{1})$ that
$$\pi_{*}( i_{a\,*}(\mathrm{Id}_{X})\cdot\alpha)=\pi_{*}(i_{a\,*}(\mathrm{Id}_{X} \cdot i_{a}^{*}(\alpha))=\pi_{*}(i_{a\,*}(i_{a}^{*}(\alpha))=\mathrm{Id}_{*}(i_{a}^{*}(\alpha))=i_{a}^{*}(\alpha).$$
By the functoriality of Chern classes, we thus find
$$\pi_{*}(i_{0\,*}(\mathrm{Id}_{X})\cdot(\ch(\tilde{G})+\ch(\tilde{E})-\ch(\tilde{F})))=(\ch(\tilde{G})+\ch(\tilde{E})-\ch(\tilde{F}))| _{X\times \{0\}}=\ch(G)+\ch(E)-\ch(F),$$
$$\pi_{*}(i_{\infty\,*}(\mathrm{Id}_{X})\cdot(\ch(\tilde{G})+\ch(\tilde{E})-\ch(\tilde{F})))=(\ch(\tilde{G})+\ch(\tilde{E})-\ch(\tilde{F}))| _{X\times \{\infty\}}=0.$$
To prove the Whitney formula, it is enough to prove the following homotopy property in the next lemma 6.
\end{proof}
\begin{mylem}
Let $i_{a}: X\hookrightarrow X\times \mathbb{P}^{1}$ be the inclusion of the compact complex submanifold $X\times \{a\}$, then $i_{a}^{*}(\alpha)$ is independent of the choice of $a$ for every $\alpha\in H_{BC}^{\bullet}(X\times\mathbb{P}^{1},\Z)$.
\end{mylem}
\begin{proof}
Since $X\times \{a\}$ is a codimension $1$ analytic set in $X\times \mathbb{P}^{1}$, its associated integral current defines a global section of $\cI_X^2$. 
Since $[X\times \{a\}]$ is of type (1,1), it projects to zero in $H^0(X, \sigma_{1,\bullet} \cD_X^{'\bullet, \bullet} \oplus \sigma_{\bullet,1} \cD_X^{'\bullet, \bullet})$.
Hence $2\pi\hspace{2pt}\sqrt{-1}[X \times \{a\}]$ defines a hypercocycle for the integral Bott-Chern complex $\cB_{1,1,\Z}^{\bullet}$. 
By the construction of the push-forward, this element represents $i_{a\,*}(\mathrm{Id}_{X})$. 
In the following we denote $i_{a\,*}(\mathrm{Id}_{X})$ as $\{2\pi\hspace{2pt}\sqrt{-1}[X \times \{a\}]\}$ (which is just the cycle class defined in the next section).  
With this notation, the end of proof of Proposition 8 gives
$$\ch(G)+\ch(E)-\ch(F)=\pi_{*}(2\pi\hspace{2pt}\sqrt{-1}(\{[X\times \{0\}]\}-\{[X\times \{\infty\}]\})\cdot(\ch(\tilde{G})+\ch(\tilde{E})-\ch(\tilde{F})))$$
and we need to show that
$$\{[X\times \{0\}]\}-\{[X\times \{\infty\}]\}=0.$$
To prove Lemma 6, we need to show that for any $p_1, p_2 \in \P^1$,
$$\{[X\times \{p_1\}]\}-\{[X\times \{p_2\}]\}=0.$$
By a biholomorphism of $\P^1$, we may assume that $p_1=0,p_2= \infty$ since $\mathrm{PGL}(2,\C)$ acts 2-transitively on projective space. 
In the following, we force on showing
$$\{[X\times \{0\}]\}-\{[X\times \{\infty\}]\}=0$$
which finishes the proof of Proposition 8 and Lemma 6 at the same time.

We denote by $z$ the parameter in $\mathbb{P}^{1}= \C \cup \{ \infty\}$ and by $[0,\infty]$ a (real) line connecting $0$ and $\infty$ in $\mathbb{P}^{1}$ (for example we can take the positive real axis). 
Then the function $\ln z$ is well defined on $\mathbb{P}^{1}\smallsetminus[0, \infty]$.
$X\times [0, \infty]$ is a real codimension one real analytic set of $X\times \mathbb{P}^{1}$, so it well defines a locally integral current.
As a current $d([X\times [0, \infty]])=-[X\times {0}]+[X\times {\infty}]$.
For any smooth form of type $(n+1,n)$ with compact support where $n$ is the complex dimension of $X$
$$\left\langle\overline{\partial}\ln z, \phi^{n+1,n}\right\rangle=-\left\langle\ln z, \overline{\partial} \phi^{n+1,n}\right\rangle=-\int_{X\times[0,\infty]^{+}-X\times[0,\infty]^{-}}\ln z\cdot \phi=-2i\pi\int_{X\times[0, \infty]}\phi.$$
The second equality is a consequence of the Stokes formula. 
It shows that  $\pr_{0,1}([X \times [0, \infty]])=-\frac{1}{2\pi i} \dbar ln(z)$. Similarly
$\pr_{1,0}([X \times [0, \infty]])=-\frac{1}{2\pi i} \d ln(\bar{z})$.
Therefore, in the space of global sections of the mapping cone $\rC(\Delta)^{\bullet}[-1](X \times \P^1)$ for $p=1,q=1$, we have
$$([X\times \{0\}]-[X\times \{\infty\}],0)=\delta (X\times [0, \infty], \frac{1}{2\pi\hspace{2pt}\sqrt{-1}}\ln \bar{z}\oplus -\frac{1}{2\pi\hspace{2pt}\sqrt{-1}}\ln z),$$
where $\delta$ is the differential of the integral Bott-Chern complex.
In other words, $[X\times \{0\}]-[X\times \{\infty\}]$ is exact, and this means that $\ch(G)+\ch(E)-\ch(F)=0$ in the integral Bott-Chern cohomology class. 
The proof of the total Chern class formula is similar.

(It would be more direct to conclude that the class of $-[X\times {0}]+[X\times {\infty}]$ is 0 in the complex Bott-Chern cohomology.
Using a resolution by currents, this is equivalent to show that as currents on $X \times \P^1$, $-[X\times {0}]+[X\times {\infty}]$ is $\d \dbar-$exact. 
However, notice that
$$-[X\times {0}]+[X\times {\infty}]= - i \d \dbar ([X] \ln |z|)$$
where we view $z$ as a meromorphic function on $\P^1$ with a single zero at $0$ and a single pole at infinity.)
\end{proof}
\begin{mylem}
Let $E$ be a vector bundle over compact complex manifold $X$.
Then we have that for total Serge current (i.e. the inverse of total Chern class)
$$s(E \oplus \cO_X)=s(E).$$
\end{mylem}
\begin{proof}
Consider the following diagram
 $$
 \begin{tikzcd}
\P(E) \arrow[d, "\pi"] \arrow[r, "i", hook] & \P(E \oplus \cO_X) \arrow[ld, "p"] \\
X.                                            &                                    
\end{tikzcd}
$$
It is enough to show that for any $k$,
$$\pi_*(c_1(\cO_{\P(E)}(1))^k)=p_*(c_1(\cO_{\P(E \oplus \cO_X)}(1))^{k+1}).$$
The integral current associated to $\P(E)$ represents the first Chern class $c_1(\cO_{\P(E \oplus \cO_X)}(1)$ since $\P(E)$ is the zero locus of a section of $\cO_{\P(E \oplus \cO_X)}(1)$.
This cycle class is also $i_* 1$.
Thus we have
$$p_*(c_1(\cO_{\P(E \oplus \cO_X)}(1))^{k+1})=p_* (c_1(\cO_{\P(E \oplus \cO_X)}(1))^k \cdot i_* 1)=p_* i_* (i^* c_1(\cO_{\P(E \oplus \cO_X)}(1))^k)=\pi_*(c_1(\cO_{\P(E)}(1))^k)$$
where the second equality follows from the projection formula and the third equality follows from the fact that $\cO_{\P(E \oplus \cO_X)}(1)|_{\P(E)}=\cO_{\P(E)}(1)$.
\end{proof}

\section{Cohomology class of an analytic set}
To prove the other axioms, we have to study the transformation of cohomology groups under what appears to be the ``wrong'' direction. For example the pull back of a cohomology class represented by the closed current associated with a cycle should morally be represented by the pull back of this current, but such pull backs are not always well defined.
In this section, given an irreducible analytic cycle $Z$ of codimension $k$ in $X$,  we will associate to it a cycle class in the integral Bott-Chern cohomology $H^{k,k}_{BC}(X,\Z)$. Then we will prove a number of elementary properties of this type of cycle classes. 
In particular, the projection formula, the transformation formula of a cycle class under a morphism will be established (Axiom B (3)). At the end, we will deduce the commutativity property of pull back and push forward by projections and inclusions, according to Axiom B (4).
The excess formula (Axiom B(5)) is a direct consequence, using the standard deformation technique of the normal bundle.


Cohomology with support is involved since cycle classes can be represented in a natural way by currents associated which the cycle. These are in fact supported in the given analytic sets, whence the appearance of cohomology with support.

 

In this section, we denote $\H_{|Z|}^{\bullet}(X, \bullet)$ or $\H_{Z}^{\bullet}(X, \bullet)$ the local hypercohomology class of some complex on $X$ with support in $Z$.
We start by defining a cycle class in the integral Bott-Chern cohomology.
This is an analogue of the cycle class in integral Deligne cohomology that has been defined in \cite{ZZ}.
As before, we denote by $\Delta: \C_X \to \sigma_p \Omega_X^{\bullet} \oplus \sigma_q \overline{\Omega_X}^{\bullet}$.

For any $p,q$, we have the following commutative diagram with exact lines 
$$  \xymatrix{
    0 \ar[r]  & \cB_{p,q,\Z}^{\bullet} \ar[r] \ar[d] & \cB_{p,q,\C}^{\bullet} \ar[r] \ar[d] &\C_X/\Z_X \ar[r] \ar[d] &0  \\
   0 \ar[r] &\Z_X \ar[r] 
   &\C_X \ar[r] & \C_X/\Z_X \ar[r]
   &0.
  }$$
The vertical morphism of complexes consists of forgetting the terms with degree $>0$.
It induces the following diagram with exact lines for $p=q=k$.
$$  \xymatrix{
    \H^{2k-1}_{|Z|}(X, \C_X/ \Z_X)\ar[r] \ar[d] & \H^{2k}_{|Z|}(X,\cB_{k,k,\Z}^{\bullet}) \ar[r] \ar[d] & \H^{2k}_{|Z|}(X,\cB_{k,k,\C}^{\bullet}) \ar[r] \ar[d] &\H^{2k}_{|Z|}(X,\C_X/\Z_X)  \ar[d] \\
   \H^{2k-1}_{|Z|}(X, \C_X/ \Z_X) \ar[r] &\H^{2k}_{|Z|}(X,  \Z_X) \ar[r] 
   &\H^{2k}_{|Z|}(X, \C_X) \ar[r] & \H^{2k}_{|Z|}(X, \C_X/ \Z_X).
  }$$
The first and fourth vertical arrow are the identity map. By the Poincar\'e duality for cohomology with support we know
$$\H^{2k-1}_{|Z|}(X, \C_X/ \Z_X) \cong \H_{2n-2k+1}(Z, \C_X / \Z_X)=0$$
where the second equality comes from the fact that the real dimension of $Z$ is $2n-2k$. 

By chasing the diagram, we know for any elements $a \in \H^{2k}_{|Z|}(X,\cB_{k,k,\C}^{\bullet})$ and $b \in \H^{2k}_{|Z|}(X,  \Z_X)$ such that their images in $\H^{2k}_{|Z|}(X, \C_X)$ are the same, then there exists a unique element in $H^{2k}_{|Z|}(X,\cB_{k,k,\Z}^{\bullet})$ such that the image of this element is $a,b$ respectively.
\paragraph{}
To define the cycle class, it is thus enough to associate the cycle two elements in $\H^{2k}_{|Z|}(X,\cB_{k,k,\C}^{\bullet})$, $\H^{2k}_{|Z|}(X,  \Z_X)$ such that their image in $\H^{2k}_{|Z|}(X, \C_X)$ is the same. 
The cycle $Z$ defines a global section in $H^0(X, \cI_X^{2k})$ so it represents an element in $\H^{2k}_{|Z|}(X,  \Z_X)$.
The inclusion $\Z_X \to \C_X$ induces in the derived category a morphism $\cI_X^{\bullet} \to \cD_X^{'\bullet}$. These two quasi-isomorphic morphisms induce the same morphism when passing to hypercohomology.
The cycle class in $\H^{2k}_{|Z|}(X,  \Z_X)$ associated with $Z$ has an image in $\H^{2k}_{|Z|}(X, \C_X)$ represented also by the integral current associated with $Z$.

 On the other hand, $\C_X$ is quasi-isomorphic to the complex $\cD_X^{'\bullet}$. 
The complex Bott-Chern complex is quasi isomorphic to the mapping cone $C(q)^{\bullet}[-1]$ with the natural map $q: \cD_X^{'\bullet} \to \sigma_{k,\bullet} \cD_X^{'\bullet,\bullet} \oplus \sigma_{\bullet,k} \cD_X^{'\bullet,\bullet}$ with a negative sign on the second component.
The integral current associated with $Z$ defines a global section of $H^0(X,\cD_X^{'\bullet})$ of bidegree $(k,k)$. And its image in $H^0(X, \sigma_{k,\bullet} \cD_X^{'\bullet,\bullet} \oplus \sigma_{\bullet,k} \cD_X^{'\bullet,\bullet})$ is 0.
This means in particular that the integration current defines a hypercocycle.
Here the hypercohomology class can be represented by this global section since the sheaf of currents is acyclic.
Hence the integration current $([Z], 0 \oplus 0)$ represents an element in $\H^{2k}_{|Z|}(X,\cB_{k,k,\C}^{\bullet})$.
Under the forgetting map $\cB_{k,k,\C}^{\bullet} \to \C_X$, its image in $\H^{2k}_{|Z|}(X,\C_X)$ can also be represented by the same integration current $[Z]$.

In conclusion, the cycle class associated with $Z$ in $\H^{2k}_{|Z|}(X,\cB_{k,k,\Z}^{\bullet})$ is exactly the class of the integral current associated with $Z$ view as an element in $H^{2k}_{|Z|}(X, \Cone(\tilde{q})^{\bullet}[-1])$ with $\tilde{q}: \cI_X^{\bullet} \to \sigma_{k,\bullet} \cD_X^{'\bullet,\bullet} \oplus \sigma_{\bullet,k} \cD_X^{'\bullet,\bullet}$. 
The image under the canonical map $\H^{2k}_{|Z|}(X,\cB_{k,k,\Z}^{\bullet}) \to \H^{2k}(X,\cB_{k,k,\Z}^{\bullet})$ defines finally the cycle class associated with $Z$ represented by the same integration current.
(This construction is already used in the proof of the Whitney formula.)
We denote in the following the cycle class associated with $Z$ as $\{[Z]\}$.
In the following, we will only consider smooth cycles.

Notice that $i_{Z*}1=\{[Z]\}$ where $1 \in H^{0,0}_{BC}(Z,\Z)$ the identity in $\oplus_{p,q} H^{p,q}_{BC}(Z,\Z)$ if $Z$ is smooth such that the Bott-Chern cohomology is well defined.
The identity in $\oplus_{p,q} H^{p,q}_{BC}(Z,\Z)$ corresponds a global constant section $1 \in \Gamma(Z,\Z_Z)$ whose image under $i_{Z*}$ in the hypercohomology is defined by locally integral current $[Z]$ by the construction of the push forward.
This global current represents the cycle class $\{[Z]\}$ on $X$. 
\paragraph{}
Now we prove some properties of cycle classes.
We start by the following lemma which expresses the push forward of a cohomology class by an arbitrary morphism in terms of the pull back and push forward of its projection, and a multiplication by the cycle class associated with the graph of the morphism.
\begin{mylem}
Let $f: X \to Y$ be a holomorphic map between complex manifolds. Assume $X$ to be compact. Let $\alpha$ be an integral Bott-Chern cohomology class. 
Denote by $\Gamma$ the graph of $f$ in $X \times Y$ and by $p_1,p_2$ the two canonical projections. Then one has
$$f_* \alpha=p_{2*}(p_{1}^* \alpha \cdot \{[\Gamma]\}).$$
\end{mylem}
\begin{proof}
This can be checked directly using the multiplication structure as in the Deligne-Beilinson complex.
The compactness condition is just used to ensure that the push-forward is well defined.
Taking $[\Gamma]$ as the global representative of the cohomology class, the cup product is induced by the wedge product between the forms and locally integral currents at the level of complexes.
We prove at the level of complexes that
$$f_*(\alpha)=p_{2*}(p_1^*(\alpha) \cup_0 [\Gamma]).$$
It suffices to check on germs on $Y$. Let $U$ be an open set of $X$ such that $U=f^{-1}(V)$ for some connected open set $V$ of $Y$.
There are two kinds of sheaves in the Deligne-Beilinson complex: locally constant sheaf in $\Z$ and sheaves of holomorphic forms.

Let $\alpha \in \Omega_Y^p(U)$. Let $\omega \in C_{(n-p,n),c}^{\infty}(U)$ be a smooth form with compact support in $U$.
Then we have
$$\langle f_* \alpha,\omega \rangle=\langle \alpha,f^*\omega \rangle=\int_U \alpha \wedge f^* \omega=\int_{\Gamma \cap p_1^{-1}(U)} p_1^* \alpha \wedge p_1^* f^* \omega$$
$$=\int_{\Gamma \cap p_1^{-1}(U)} p_1^* \alpha \wedge p_2^* \omega=\langle p_{2*}(p_1^*(\alpha) \cup_0 [\Gamma]), \omega \rangle.$$
Notice that $p_1$ induces a biholomorphism between $\Gamma\cap p_1^{-1}(U) $ and $U$.

For $c \in \Z_X(U)$, its image under $f_*$ via the quasi-isomorphism is the local integral current $c f_*[U]$. The equality at the level of complexes is just
$$c f_*[U]=p_{2*}(c[\Gamma \cap p_2^{-1}(U)])=p_{2*}(p_1^*(c) \cup_0 [\Gamma]).$$
Passing to hypercohomology gives the desired equality.
\end{proof}

As in \cite{Gri}, we have the following property. It is a combination of the above lemma and the pull back of the cycle class under a closed immersion.
\begin{myprop} {\it
Let $f: {X} \to {Y}$ be a surjective proper map
between compact manifolds, and let $D$ be a smooth
divisor of $Y$. 
We denote $f^* D=m_{1}\tilde{D}_1+\cdots +m_{N}\tilde{D}_{N}$. Let
$\tilde{f}_i: \tilde{D}_{i} \to {D}$ ($1 \leq i \leq N $) be
the restriction of $f$ to
$\tilde{D}_{i}$ with smooth $\tilde{D}_{i}$. Then we have in integral Bott-Chern cohomology
$$
f^* i_{D*}=\sum _{i=1}^{N} m_{i}\,
i_{\tilde{D}_{i}*} \tilde{f}_{i}^{*}.
$$}
\end{myprop}
\begin{proof}
The proof is identical to the case of the Deligne complex. For self-containedness, we give briefly the details to indicate where Proposition 3 is needed and used.
The idea consists of passing to the graph and using the above lemma.
Since all spaces are compact, the push-forward is always well-defined. Let $\Gamma$ be the graph of $i_D: D \hookrightarrow Y$ and let $\tilde{\Gamma'}_i$ be the graph of $i_{\tilde{D'}_i}: \tilde{D'}_i \hookrightarrow X$.
We denote all terms involving $X$ with a prime symbol $'$ and
all other terms without that symbol.
By definition, $[\Gamma'_i]:=(\tilde{f}_i,\id)_* [\tilde{\Gamma}'_i]$ as current which induces as cycle class $\{[\Gamma'_i]\}=(\tilde{f}_i,\id)_* \{[\tilde{\Gamma}'_i]\}$.
$[\Gamma'_i]$ is supported in the image of $(\tilde{f}_i,\id)$.
We denote by $p_j$ ($j=1,2$) the natural projections of $D \times Y$, by $p'_j$ projections of $D \times X$, and by $\tilde{p}'_{j,i}$ projections of $\tilde{D}_i \times X$.

In terms of currents, we have $(\id,f)^*[\Gamma]=\sum_{i=1}^N m_i [\Gamma'_i]$. We can prove the cycle class equality $(\id,f)^*\{[\Gamma]\}=\sum_{i=1}^N m_i \{[\Gamma'_i]\}$ in integral Bott-Chern cohomology as in Lemma 10.
In fact, it can be reduced from the corresponding equality in Deligne cohomology proven in Corollary 7.7 \cite{EV} sicne the natural morphism of Chow group to integral Deligne cohomologies is compatible with pull back.
Then we have
$$f_*i_{D*} \alpha=f^*p_{2*}(p_1^*\alpha\cdot\{[\Gamma]\})=p'_{2*}(\id,f)^* (p_1^*\alpha\cdot\{[\Gamma]\}) $$
$$=p'_{2*}((\id,f)^* p_1^*\alpha\cdot(\id,f)^*\{[\Gamma]\})=\sum_{i=1}^N m_i p'_{2*}(p_1^{'*} \alpha\cdot\{[\Gamma'_i]\})$$
$$=\sum_{i=1}^N m_i p'_{2*}(\tilde{f_i},\id)_*((\tilde{f_i},\id)^*p_1^{'*} \alpha\cdot\{[\tilde{\Gamma}'_i]\})=\sum_{i=1}^N m_i \tilde{p'}_{2,i*}(\tilde{p'}_{1,i}^*\tilde{f_i}^{*} \alpha\cdot\{[\tilde{\Gamma}'_i]\})$$
$$=\sum _{i=1}^{N} m_{i}\,
i_{\tilde{D}_{i}*} \tilde{f}_{i}^{*} \alpha.$$
The first equality uses Lemma 8. The second formula uses Proposition 3 for $f \circ p'_2 =p_2 \circ (\id,f)$. The third equality uses the fact that pull-back is a ring morphism. The fourth equality uses the fact that $p'_1=p_1 \circ (\id,f)$. The fifth equality uses the projection formula. The sixth equality uses the fact that $\tilde{p'}_{2,i}=(\tilde{f_i},\id) \circ p'_2$ and $\tilde{f_i}\circ \tilde{p'}_{1,i}=p'_1 \circ (\tilde{f_i},\id)$. The last equality uses another time Lemma 8.
The surjectivity of $f$ is just used to ensure that the pull-back of a divisor is a divisor.
\end{proof}
We give an easy generalisation of a lemma in \cite{Sch}. It gives the expected relation between the integral Bott-Chern cohomology and the Deligne cohomology.
In particular, one can reduce the relevant properties of cycle classes in the integral Bott-Chern cohomology to the Deligne complex case, when they only involve the group structure.
\begin{mylem}
For any $p \geq 1$, we have a $\Z$-module isomorphism
$$H^{p,p}_{BC}(X,\Z) \simeq H^{2p}_D(X,\Z(p)) \oplus \overline{\H^{2p-1}(X, \Omega_{<p}^{\bullet})}.$$
Moreover, via the isomorphism, for any proper cycle $Z$ in $X$, the cycle class $\{[Z]\}_{BC}$ associated with $Z$ in the integral Bott-Chern cohomology corresponds to $(\{[Z]\}_D,0)$, where $\{[Z]\}_D$ is the cycle class associated with $Z$ in the Deligne cohomology.
\\
This isomorphism is functorial with respect to pull backs.
\end{mylem}
\begin{proof}
We have the short exact sequence
$$0 \to \overline{\Omega}^{\bullet}_{<p}[1] \to \cB_{p,p,\Z}^{\bullet} \to \cD(p)^{\bullet} \to 0.$$
We can prove as shown in \cite{Sch} that the short exact sequence is in fact split, so that we have an abelian group isomorphism 
$$H^{p,p}_{BC}(X,\Z) \simeq H^{2p}_D(X,\Z(p)) \oplus \overline{\H^{2p-1}(X, \Omega_{<p}^{\bullet})}$$
by taking the hypercohomology.
We have to transform the complex involving smooth forms into a cone complex involving currents. These complexes are quasi-isomorphic, so that the splitting induces a morphism of complexes in the derived category.
However, we want to modify that splitting to relate the cycle classes in our different cohomology theories (respectively Deligne and integral Bott-Chern).
\\
Let $A$ be the matrix 
$$\begin{pmatrix}
\frac{1}{2}& -\frac{1}{2} \\
\frac{1}{2} & \frac{1}{2}
\end{pmatrix}$$ 
We use the construction for $A$ given in the next remark which shows that the integral Bott-Chern complex is quasi-isomorphic to $\Cone(\cI_X^{\bullet} \xrightarrow{(1,0)} \sigma_{p, \bullet}\cD^{\bullet,\bullet}_X \oplus \sigma_{ \bullet,p}\cD^{\bullet,\bullet}_X )[-1]$. 
The Deligne complex is quasi-isomorphic to $\Cone(\cI_X^{\bullet} \xrightarrow{\pr_{p, \bullet}} \sigma_{p, \bullet}\cD^{\bullet,\bullet}_X )[-1]$.
There exists a splitting morphism given by for any element $(a,b)  \in \cI^k_X \oplus \sigma^{k-1}_{p, \bullet} \cD^{\bullet,\bullet}_X $ by
$$F:\Cone(\cI_X^{\bullet} \xrightarrow{\pr_{p, \bullet}} \sigma_{p, \bullet}\cD^{\bullet,\bullet}_X)[-1] \to \Cone(\cI_X^{\bullet} \xrightarrow{(1,0)} \sigma_{p, \bullet}\cD^{\bullet,\bullet}_X \oplus \sigma_{ \bullet,p}\cD^{\bullet,\bullet}_X )[-1]$$
$$(a,b) \mapsto (a, b, 0).$$ 
We verify that it is a morphism of complexes:
$$F(d(a,b))=F(-da,\pr_{p, \bullet}a+\d b)=(-da,\pr_{p, \bullet}a+\d b, 0)$$
$$\kern58pt{}=d(F(a,b))=d(a, b, 0)=(-da,\pr_{p, \bullet}a+\d b, \dbar 0).$$
Via this splitting isomorphism the cycle class associated with an analytic set $Z$ is the cohomology class represented by $[Z]$ and $([Z],0)$ respectively. 
Thus the image of the cycle class $\{[Z]\}_D$ under $F$ is $\{[Z]\}_{BC}$.

The functoriality comes from the functoriality of the construction given in the remark.
\end{proof}
\begin{myrem}
{\rm 
The sign in the definition of the integral Bott-Chern complex is unimportant for the group structure of the integral Bott-Chern cohomology when $p=q$. In fact, up to an isomorphism of abelian group, we can change the vector $(1,-1)$ to be any non zero vector in $\C^2$. To do it, we need the following construction. 

Recall that the integral Bott-Chern complex is $\Cone(\Z \xrightarrow{(+,-)} \Omega_{<p}^{\bullet} \oplus \overline{\Omega}^{\bullet}_{<p})[-1]$ the mapping cone of the morphism $\Z \xrightarrow{(+,-)} \Omega_{<p}^{\bullet} \oplus \overline{\Omega}^{\bullet}_{<p}$.
Let $A \in GL(2, \R)$ be any invertible matrix.
We denote by $a_{ij} (1 \leq i,j \leq 2)$ the elements of $A$. Then we have the following isomorphism of $\Z_X$-complex $\Omega_{<p}^{\bullet} \oplus \overline{\Omega}^{\bullet}_{<p}$. 
For any $k$, $(\omega_1, \omega_2) \in \Omega^k \oplus \overline{\Omega}^k$ sends to $(a_{11}\omega_1 +a_{12} \overline{\omega_2},a_{21}\overline{\omega_1} +a_{22} \omega_2)$.
The conjugation transforms the holomorphic forms to the anti-holomorphic forms and vice versa.
(In fact it is $\R_X$-morphism not $\C_X$-morphism.)
The inverse morphism is induced by the matrix $A^{-1}$. 
\\
Via this isomorphism of complex of $\Z_X$-sheaves, the integral Bott-Chern complex is isomorphic to
$$\Cone(\Z \xrightarrow{A(1,-1)^t} \Omega_{<p}^{\bullet} \oplus \overline{\Omega}^{\bullet}_{<p})[-1].$$
For any vector $(a,b) \in \C^2$, if we choose adequately $A$ so that $(a,b)^t =A (1,-1)^t$,
the integral Bott-Chern complex is isomorphic to $\Cone(\Z \xrightarrow{(a,b)} \Omega_{<p}^{\bullet} \oplus \overline{\Omega}^{\bullet}_{<p})[-1]$, which induces an isomorphism by passing to hypercohomology.
This construction is functorial with respect to pull-backs, since the pull-back by a holomorphic map preserves the holomorphic forms and the anti-holomorphic forms.

This construction does not work for complex Bott-Chern cohomology since the isomorphism we have constructed is not complex linear.

The integral Bott-Chern complex is quasi-isomorphic to $\Cone(\cI_X^{\bullet} \xrightarrow{\Delta} \sigma_{p, \bullet}\cD^{\bullet,\bullet}_X \oplus \sigma_{ \bullet,p}\cD^{\bullet,\bullet}_X )[-1]$.
Via this quasi-isomorphism, the above construction gives an isomorphism of complexes
$$F: \Cone(\cI_X^{\bullet} \xrightarrow{\Delta} \sigma_{p, \bullet}\cD^{\bullet,\bullet}_X \oplus \sigma_{ \bullet,p}\cD^{\bullet,\bullet}_X )[-1] \to \Cone(\cI_X^{\bullet} \xrightarrow{A(1,-1)^t} \sigma_{p, \bullet}\cD^{\bullet,\bullet}_X \oplus \sigma_{ \bullet,p}\cD^{\bullet,\bullet}_X )[-1]$$
sending $(a,b,c)$ to $(a, a_{11}b+a_{12}\overline{c},a_{21}\overline{b}+a_{22}c)$.
Here $A(1,-1)^t$ is the composition of $\Delta$ with the morphism given as in the above construction for $\sigma_{p, \bullet}\cD^{\bullet,\bullet}_X \oplus \sigma_{ \bullet,p}\cD^{\bullet,\bullet}_X$ and $A$.
Concretely for any $k$, the differential of $T \in {\cI_X}^{k}$ sends to $(a_{11}\pr_{p,\bullet}T -a_{12} \pr_{p,\bullet}T,a_{21}\pr_{\bullet,p}T -a_{22} \pr_{\bullet,p}T)$ with value in $ \sigma_{p, \bullet}\cD^{\bullet,k}_X \oplus \sigma_{ \bullet,p}\cD^{k,\bullet}_X$.
We check that $A$ induces a morphism of complexes.
$$F(d(a,b,c))=F(-da, \pr_{p, \bullet}a+ \d b, -\pr_{\bullet,p}\overline{a}+\dbar c)
\kern200pt{}$$
$$\kern70pt{}=(-da, a_{11}\pr_{p,\bullet}a+a_{11}\d b -a_{12}\overline{\pr_{\bullet,p}\overline{a}}+a_{12}\d \overline{c}, a_{21}\overline{\pr_{p,\bullet}a}+a_{21} \dbar b-a_{22}\pr_{\bullet,p}\overline{a}+a_{22}\dbar c).$$
$$d(F(a,b,c))=d(a, a_{11}b+a_{12}\overline{c},a_{21}\overline{b}+a_{22}c)\kern220pt{}$$
$$\kern70pt{}=(-da, a_{11}\pr_{p,\bullet}a+a_{11}\d b -a_{12}\pr_{p,\bullet}a+a_{12}\d \overline{c}, a_{21}\pr_{\bullet,p}\overline{a}+a_{21} \dbar b-a_{22}\pr_{\bullet,p}\overline{a}+a_{22}\dbar c).$$
In particular, since the cycle class associated with an analytic set $Z$ is represented by the global section $([Z],0 \oplus 0)$ where $[Z]$ is the current associated with $Z$, its image under the isomorphism is represented by the same section for any matrix $A$.
}
\end{myrem}
Now we return to the transformation of a cycle class under a morphism in the integral Bott-Chern cohomology.
\begin{mylem}
Let $X$ be any complex manifold, $Y$ and $Z$ be compact submanifolds of $X$ which intersect transversally and let $W=Y \cap Z$. Let $i_Y: Y \to X$ be the inclusion.
Then we have in the integral Bott-Chern cohomology the equality
$$i_Y^* \{[Z]\}=\{[W]\}.$$
\end{mylem}
\begin{proof}
In this proof we denote $\{[Z]\}_{BC}$ for the cycle class associated with an analytic set $Z$ in the integral Bott-Chern cohomology and $\{[Z]\}_{D}$ for the corresponding class in the Deligne cohomology.
Via the isomorphism given in Lemma 9 and the functoriality, the equality $i_Y^* \{[Z]\}_{BC}=\{[W]\}_{BC}$ is equivalent to the equality $i_Y^* \{[Z]\}_D=\{[W]\}_D$. The proof in the Deligne complex case is given in the following via the Bloch cycle classes by Proposition 7.5 of \cite{EV}.
One may refer to \cite{Wu20} for more details, or more explicit versions of some
of the proofs.
\end{proof}

In fact Lemma 10 gives as a special case the following proposition, which translates into the equality $i_{Y}^* i_{Z*}1=i_{W/Y*}i^*_{W/Z}1$.
\begin{myprop} {\it
Consider the following commutative diagram, where
  $Y$ and $Z$ are compact and intersect transversally with $W=Y \cap Z$:
$$
\xymatrix{
W  \ar@{^{(}->}[r]^{i_{W/Y}}
\ar@{_{(}->}[d]_{i_{W/Z}}
&Y\ar@{_{(}->}[d]^{i_{Y}}
\\
Z \ar@{^{(}->}[r]_{i_{Z}}&X
}
$$
Then we have in the integral Bott-Chern cohomology
$i_{Y}^* i_{Z*}=i_{W/Y*}i^*_{W/Z}$.}
\end{myprop}
\begin{proof}
The same proof in \cite{Gri} holds.
\end{proof}
The transversality condition is necessary in the above proposition. Indeed, if we take $Y=Z=W$, the morphism $i_{Y}^* i_{Y*}$ is not equal to the identity. To calculate it, we need the following excess formula. In the reverse direction, the formula is far easier. For any smooth submanifold $Z$ of $X$ and any cohomology class $\alpha$ on $X$ we have
$$i_{Z*} i_Z^* \alpha=\alpha \cdot \{[Z]\}.$$
This can be derived from the projection formula, which implies
$$i_{Z*} i_Z^* \alpha=i_{Z*}(i_Z^* \alpha \cdot 1)=\alpha \cdot i_{Z*}1=\alpha \cdot \{[Z]\}.$$
\begin{myprop} {\it 
If $Y$ is a smooth hypersurface of $X$ with $X$ a compact complex manifold, then for any $\alpha$ an integral Bott-Chern cohomological class,
$$i_Y^* i_{Y*} \alpha=\alpha \cdot c_1(N_{Y/X}).$$}
\end{myprop}
\begin{proof}
The same proof in \cite{Gri} holds using the deformation of the normal cone (cf. \cite{Fu} chap V). 
\end{proof}
\section{Transformation under blow-up}
In this part, we want to show that the integral Bott-Chern class satisfies the rest of the axioms B in 
\cite{Gri} (see Axiom B (5)(6)(7) in Introduction). 

To start with, we prove the transformation formula of the integral Bott-Chern cohomology under blow up.
The closed immersions, projections and blow ups are the most elementary morphisms in the description of Serre's proof of Riemann-Roch-Grothendieck formula.
In fact, by considering the graph, any projective morphism can be written as a composition of a closed immersion and a projection. 
By devissage, we reduce the general closed immersion to the case of closed immersion of a smooth hypersurface.  
To perform this reduction, we need to blow up submanifolds, and thus a study of the cohomology of blow ups is required. To do this, we will need the following version for Dolbeault cohomology groups stated in \cite{RYY}.
\begin{mythm}
Let $X$ be a compact complex manifold with $\dim_{\C} X = n$ and $Y \subset X$ a closed complex submanifold of complex codimension $r \geq 2$. Suppose that $p : \tilde{X}\to X$ is the
blow-up of $X$ along $Y$. We denote by $E$ the exception divisor and by $i: Y \to X$, $j: E \to \tilde{X}$ the inclusions, by $q: E \to Y$ the restriction of $p$ on $E$. Then for any $0 \leq l,m \leq n$, there are isomorphisms
$$ H^{l,m}_{\dbar }(X) \oplus \oplus_{i=0}^{r-2}H_{\dbar}^{l-i-1,m-i-1}(Y) \xrightarrow{(p^*,j_* \circ c_1(\cO_{\P(N_{Y/X})}(1))^i \wedge \circ q^*  )} H^{l,m}_{\dbar}(\tilde{X}),$$
$$ H^{l,m}_{\d }(X) \oplus \oplus_{i=0}^{r-2}H_{\d}^{l-i-1,m-i-1}(Y) \xrightarrow{(p^*,j_* \circ c_1(\cO_{\P(N_{Y/X})}(1))^i \wedge \circ q^*  )} H^{l,m}_{\d}(\tilde{X}).$$
\end{mythm}
\begin{proof}
Let us start with some explanations on Gysin morphism.
For any $k \in \N$, $\Omega^k$ is quasi-isomorphic to the complex of smooth forms $C^{k, \bullet}$ which is also quasi-isomorphic to the complex of currents $\D'^{k,\bullet}$ by Dolbeault–Grothendieck lemma.
 Using the resolution by currents, as in the integral Bott-Chern cohomology case, one can define a functional Gysin morphism in Dolbeault cohomologies.
 In this case, if $Y$ is a codimension $r$ smooth submanifold of $X$, the image of associated Bloch cycle class of $Y$ defines a cycle class in $H^r(X,\Omega^r_X)$ (in fact in $H^r_Y(X, \Omega^r_X)$).
The same proof in \cite{Gri} proves analogue of Proposition 11 for Dolbeault cohomologies.
 
The first statement follows from the main theorem of \cite{RYY}. 
Note that by their results, both sides have the same complex dimension.
To prove the isomorphism, it is enough to show that it is injective.
Assume that $p^* \alpha + \sum_{i=0}^{r-2} j_*(q^* \beta_i \wedge c_1(\cO_{\P(N_{Y/X})}(1))^i)=0$.
Taking $j^* $ by analogue of Proposition 11 for Dolbeault cohomologies gives
$$q^* i^* \alpha +\sum_{i=0}^{r-2} q^* \beta_i \wedge  c_1(\cO_{\P(N_{Y/X})}(1))^{i+1}=0.$$
By Theorem 3, $\beta_i=0 (\forall i), i^* \alpha=0$.
Taking $p_*$ on the original equation gives $\alpha=p_* p^* \alpha=0$ by projection formula which finishes the proof.
The second statement uses the fact that
$$H^{l,m}_{\d}(X)=\mathrm{ker}\{\d: \Gamma(X, C^{l,m}_{\infty}) \to \Gamma(X, C^{l+1,m}_{\infty})\}/\mathrm{Im}\{\d: \Gamma(X, C^{l-1,m}_{\infty}) \to \Gamma(X, C^{l,m}_{\infty}) \}$$
$$=\overline{\mathrm{ker}\{\dbar: \Gamma(X, C^{m,l}_{\infty}) \to \Gamma(X, C^{m,l+1}_{\infty})\}/\mathrm{Im}\{\dbar: \Gamma(X, C^{m,l-1}_{\infty}) \to \Gamma(X, C^{m,l}_{\infty}) \}}=\overline{H^{m,l}_{\dbar}(X)}.$$
Now the second statement comes from the first statement.
\end{proof}
We also need the classical analogue for integral coefficient cohomology (cf. \cite{GH}, page 603 or Theorem 7.31 \cite{Voi07}) by using the Mayer-Vietoris sequence involving a tubular neighbourhood of $Y$.
\begin{mylem}
Let $X$ be a compact complex manifold with $\dim_{\C} X = n$ and $Y \subset X$ a closed complex submanifold of complex codimension $r \geq 2$. Suppose that $p : \tilde{X}\to X$ is the
blow-up of $X$ along $Y$. We denote by $E$ the exception divisor and by $i: Y \to X$, $j: E \to \tilde{X}$ the inclusions, by $q: E \to Y$ the restriction of $p$ on $E$. Then for any $k$ there is an isomorphism
$$ H^{k}_{}(X,\Z) \oplus \oplus_{i=0}^{r-2}H^{k-2i-2}(Y, \Z) \xrightarrow{(p^*,j_* \circ c_1(\cO_{\P(N_{Y/X})}(1))^i \wedge \circ q^*  )} H^{k}_{}(\tilde{X},\Z).$$
\end{mylem}
\begin{proof}
Note that in the book \cite{Voi07}, the theorem is stated with the assumption that $X$ is K\"ahler. 
However, its proof only uses the K\"ahler condition in Lemma 7.28 \cite{Voi07} to conclude that $p^*$ is injective which can be replaced by an analogue of Lemma 3 with $p=q=0$ but with degree $n$ hypercohomology.

Note that for the rational coefficients case, the proof of Theorem 4 with \cite{GH} can more easily conclude the proof.
\end{proof}

Using these results, we can prove by induction an analogous result for integral Bott-Chern cohomology.
\begin{myprop} {\it
Let $X$ be a compact complex manifold with $\dim_{\C} X = n$ and $Y \subset X$ a closed complex submanifold of complex codimension $r \geq 2$. Suppose that $p : \tilde{X}\to X$ is the
blow-up of $X$ along $Y$. We denote by $E$ the exception divisor and by $i: Y \to X$, $j:E \to \tilde{X}$ the inclusions, by $q: E \to Y$ the restriction of $p$ on $E$. Then for any $k,l,m$ there is an isomorphism
$$ \H^{k}_{ }(X, \cB^{\bullet}_{l,m,\Z}) \oplus \oplus_{i=0}^{r-2}\H_{}^{k-2i-2}(Y, \cB^{\bullet}_{l,m,\Z}) \xrightarrow{(p^*,j_* \circ c_1(\cO_{\P(N_{Y/X})}(1))^i \wedge \circ q^*  )} \H^{k}_{}(\tilde{X}, \cB^{\bullet}_{l,m,\Z}).$$
In paricular, there is an isomorphism
$$j^*: \H^{k}_{}(\tilde{X},\cB^{\bullet}_{l,m,\Z})/p^*\H^{k}_{}(X,\cB^{\bullet}_{l,m,\Z}) \cong \H^{k}_{}(E,\cB^{\bullet}_{l,m,\Z})/q^*\H^{k}_{}(Y,\cB^{\bullet}_{l,m,\Z}).$$}
\end{myprop}
\begin{proof}
The short exact sequence 
$$0 \to \Omega^{l+1}[-l-1] \to \cB_{l+1,m,\Z}^{\bullet} \to \cB_{l,m,\Z}^{\bullet} \to 0$$ 
induces a commutative diagram
\tiny{
\[ \begin{tikzcd}
H^{k-l-1,l+1}_{\dbar}(X)\oplus \oplus_{i=0}^{r-2} H^{k-l-1-i,l+1-i}_{\dbar}(Y) \arrow{r}{} \arrow{d}{} 
&  \H^{k}_{ }(X, \cB^{\bullet}_{l+1,m,\Z}) \oplus \oplus_{i=0}^{r-2}\H_{}^{k-2i-2}(Y, \cB^{\bullet}_{l+1,m,\Z})\arrow{r}{} \arrow{d}{} 
&  \H^{k}_{ }(X, \cB^{\bullet}_{l,m,\Z}) \oplus \oplus_{i=0}^{r-2}\H_{}^{k-2i-2}(Y, \cB^{\bullet}_{l,m,\Z}) \arrow{r}{}\arrow{d}{}&\cdots\\%
H^{k-l-1,l+1}_{\dbar}(\tilde{X}) \arrow{r}{} & \H^{k}_{ }(\tilde{X}, \cB^{\bullet}_{l+1,m,\Z}) \arrow{r}{} 
&\H^{k}_{ }(X, \cB^{\bullet}_{l,m,\Z})  \arrow{r}{} &\cdots
\end{tikzcd}
\]}
\normalsize{
By the five lemma and Theorem 4,
} 
one can reduce the proof to the case $l=0$ by induction.
Then the short exact sequence
$$0 \to \overline{\Omega^{m+1}[-m-1]} \to \cB_{0,m+1,\Z}^{\bullet} \to \cB_{0,m,\Z}^{\bullet} \to 0$$ 
induces a commutative diagram
\tiny{
\[ \begin{tikzcd}
H^{m+1,k-m-1}_{\dbar}(X)\oplus \oplus_{i=0}^{r-2} H^{m+1-i,k-m-1-i}_{\dbar}(Y) \arrow{r}{} \arrow{d}{} 
&  \H^{k}_{ }(X, \cB^{\bullet}_{0,m+1,\Z}) \oplus \oplus_{i=0}^{r-2}\H_{}^{k-2i-2}(Y, \cB^{\bullet}_{0,m+1,\Z})\arrow{r}{} \arrow{d}{} 
&  \H^{k}_{ }(X, \cB^{\bullet}_{0,m,\Z}) \oplus \oplus_{i=0}^{r-2}\H_{}^{k-2i-2}(Y, \cB^{\bullet}_{0,m,\Z}) \arrow{r}{}\arrow{d}{}&\cdots\\%
H^{m+1,k-m-1}_{\dbar}(\tilde{X}) \arrow{r}{} & \H^{k}_{ }(\tilde{X}, \cB^{\bullet}_{l+1,m,\Z}) \arrow{r}{} 
&\H^{k}_{ }(X, \cB^{\bullet}_{l,m,\Z})  \arrow{r}{} &\cdots
\end{tikzcd}
\]}
\normalsize{%
By the five lemma and Theorem 4 again,
} 
one can reduce the proof to the case $l=0$, $m=0$ by induction. This is done directly by Lemma 11.
\end{proof}

A direct application of the proposition is the following general excess formula compared to Proposition 11.
\begin{myprop} {\it
With the same notation in the above proposition, if $F$ is the excess conormal bundle on $E$ defined by the exact sequence
$$0 \to F \to q^*N^*_{Y/X} \to N^*_{E/\tilde{X}}\to 0,$$
one has the following excess formula for any cohomology class $\alpha$ on $Y\;:$
$$p^* i_* \alpha=j_* (q^* \alpha \cdot c_{d-1}(F^*)).$$}
\end{myprop}
\begin{proof}
Define $\beta=j_*(q^* \alpha \cdot c_{d-1}(F^*))$. By the excess formula for a line bundle, we have
$$j^* \beta=[q^* \alpha \cdot c_{d-1}(F^*)]\cdot c_1(N_{E/\tilde{X}})=q^* \alpha \cdot q^*(c_d(N_{Y/X})).$$
The second equality uses the Whitney formula for Chern class of vector bundles. 
Hence $j^* \beta \in \mathrm{Im}(q^*)$ and by the above Proposition we know $\beta=p^* \gamma$ for some cohomology class on $X$.
So $p_* \beta=p_* p^* \gamma=\gamma$ where the second equality uses $p_* p^*=\id$ proven in the second section.
Then we have
$$\beta=p^* p_* \beta=p^* p_* j_* (q^* \alpha \cdot c_{d-1}(F^*))=p^* i_* q_* (q^* \alpha \cdot c_{d-1}(F^*))$$
$$=p^* i_* (\alpha \cdot q_* c_{d-1}(F^*))=p^* i_* \alpha.$$
The first equality on the second line uses the projection formula.
The last equality uses the fact that $q_* c_{d-1}(F^*)=1$, as follows from the next lemma.
\end{proof}
\begin{mylem}
Let $G \to X$ be a vector bundle of rank $r$ which induces $\pi: \P(G) \to X$. Let $H$ be the vector bundle defined by the exact sequence
$$0 \to H \to \pi^* G \to \cO_{\P(G)}(1) \to 0.$$
Then we have $\pi_*(c_{r-1}(H))=(-1)^{r-1}$.
\end{mylem}
\begin{proof}
We start the proof for the complex Bott-Chern cohomology such that the cohomology class can be represented by global differential forms.
By the Whitney formula for the total Chern class, $c(\pi^* G)=c(H) \cdot c(\cO_{\P(G)}(1))$.
We denote $h:=c_1(\cO_{\P(G)}(1))$. Then
$$c(H)=c(\pi^* G)(1+h)^{-1}=(1+c_1(\pi^* G) +\cdots + c_r(\pi^* G))(1-h+h^2+\cdots).$$
The element of degree $r-1$ on two sides is $c_{r-1}(H)=(-1)^{r-1} h^{r-1} +(-1)^{r-2}h^{r-2}c_1(\pi^* G)+\cdots +c_{r-1}(\pi^* G)$. 
$\pi_*$ is given by integration along the fibre direction. By degree reason,
$\pi_* c_{r-1}(H)=(-1)^{r-1} \pi_* h^{r-1}=(-1)^{r-1}$.
The integration can be calculated by a metric on $\cO_{\P(G)}(1)$ induced by a smooth Hermitian metric on $G$. 
This finishes the proof of the complex case. 

Since the equality is taken in $H^{0,0}_{BC}(X,\Z)=H^0(X,\Z)\cong \Z$ which is a lattice in $H^{0,0}_{BC}(X,\C)=H^0(X,\C)\cong \C$. 
We deduces the integral case from the complex one.
\end{proof}
Everything we have done also works for rational Bott-Chern cohomology. In \cite{Gri}, Grivaux shows that as soon as one has a good intersection theory for some cohomology theory, one can use the Riemann-Roch-Grothendieck formula to construct the Chern class of a coherent sheaf by an induction on dimension.
The last axiom that remains to be proven is the Hirzebruch–Riemann–Roch theorem.
It can be reduced to the case of the Deligne complex by the following observation made in lemma 7.2 of \cite{Sch}.
\begin{mylem}
Let $X$ be a compact K\"ahler manifold. Then for any $p \in \N^*$ and $k \in \N$ we have
$$\H^k(X, \Omega^{\bullet}_{<p}) \cong \oplus_{r+s=k, r< p}H^{r,s}(X,\C).$$
\end{mylem}
Since $\P^n$ is K\"ahler, the lemma gives the complete description of the integral Bott-Chern cohomology for the projective spaces.
\begin{myprop} {\it
The natural morphism $\oplus_k H_{BC}^{k,k}(\P^n, \Z) \to \oplus_p H^{2p}_D(\P^n,\Z(p))$ induces an isomorphism of rings. In particular,
the Hirzebruch–Riemann–Roch theorem holds for integral Bott-Chern cohomology.}
\end{myprop}
\begin{proof}
By Lemma 13, we have for any $p \in\N^*$
$$\H^{2p}(\P^n,\Omega_{<p}^{\bullet})=0 \to H^{p,p}_{BC}(\P^n,\Z) \to H^{2p}_D(\P^n, \Z(p)) \to \H^{2p+1}(\P^n,\Omega_{<p}^{\bullet})=0.$$
The second morphism is the natural morphism from Bott-Chern cohomology to Deligne cohomology which is in fact an isomorphism shown by the exact sequence.
For $p=0$, it is also an isomorphism since the complexes are the same. 
Since the natural morphism from Bott-Chern cohomology to Deligne cohomology is a ring morphism, we have the first statement.
\end{proof}
\begin{myrem}
{\rm As far as we know, it seems that Grivaux's method does not work for constructing Chern classes of a coherent sheaf in the integral Bott-Chern cohomology, as opposed to the rational cohomology.
The main reason is that the Chern characteristic class is additive but the total Chern class is multiplicative, and switching from one to the other involves denominators. The proof given in \cite{Fu} for the Riemann-Roch-Grothendieck formula in the context of coherent sheaves and the Chow ring reduces to proving that the Riemann-Roch-Grothendieck formula holds for vector bundles. 
The additivity of the Chern characteristic class and the nature of the formula ensure that after proving the special case of bundles, the Riemann-Roch-Grothendieck formula will also be valid for coherent sheaves on projective manifolds.
However, one needs the projectivity condition to ensure that the Grothendieck group of coherent sheaves and the Grothendieck group of vector bundles are the same. 

There exists an analogue of the ``integral'' Riemann-Roch-Grothendieck formula given in \cite{Jou}. In this work, Jouanolou proved that for a closed embedding $f: X \to Y$ of non-singular varieties of codimension $d$ and for any vector bundle of rank $e$ on $X$, then the total Chern class in Chow groups satisfies
$$c(f_* E)=1+f_*(P(c_1(N), \cdots, c_d(N),c_1(E), \cdots, c_e(E)))$$
where $N$ is the normal bundle and $P$ is some universal polynomial depending only on $d,e$.
This formula does not work directly for coherent sheaves by simply replacing $e$ with the generic rank of the coherent sheaf involved, even in the projective case. This is caused by the lack of additivity and the appearance of polynomials. As a consequence, a different choice of the values of $e$ will give a completely different class. As a matter of fact, a coherent sheaf can carry in its Chern classes some information that extend to degrees beyond its generic rank. At this point, there does not seem to exist a similar integral Riemann-Roch-Grothendieck formula for coherent sheaves. 

An easy counter example is obtained by considering $f: \P^2 \to \P^3$ and $\cF=\cO_{\P^2}/ m_0$. The left hand side is equal to 
$c(\cO_{\P^3}/ m_0)=\frac{c(\cO_{\P^3})}{c(m_0)}=1-c_1(\cO_{\P^3}(1))^3$,
but the right hand of the universal polynomial with $d=1,e=1$ where 1 is the generic rank of $\cO_{\P^3}/ m_0$ gives $1+f_*P(c_1(N),c_1(\cO_{\P^2}/ m_0))=1+f_*P(c_1(\cO_{\P^2}(1)),c_1(\cO_{\P^2}/ m_0))=1+f_* \big( \frac{1}{1+c_1(\cO_{\P^2}/ m_0)-c_1(\cO_{\P^2}(1))}-1 \big)=1+c_1(\cO_{\P^3}(1))^2+c_1(\cO_{\P^3}(1))^3$.
The same example shows that the formula is not valid when we taking $e$ to be the largest number such that the Chern class is not trivial.
We do not know whether there are any substitutes of the Riemann-Roch-Grothendieck formula used in Grivaux's induction argument, that would be capable of defining Chern classes in integral Bott-Chern cohomology.
}
\end{myrem}

\end{document}